\documentclass[11pt,twoside,a4paper]{amsart}
\usepackage{amssymb}
\date{\today}

\date{\today}
\usepackage{geometry}
\def\a{{\mathfrak a}}

\def\1{{\bf 1}}

\def\End{{\rm End}}

\def\deg{\text{deg}\,}

\def\w{\wedge}

\def\dbar{\bar\partial}

\def\C{{\mathbb C}}
\def\w{{\wedge}}
\def\P{{\mathbb P}}

\def\A{{\mathcal A}}
\def\Be{{\mathcal B}}
\def\B{{\mathbb B}}
\def\V{{\mathcal V}}

\def\Fitt{\text{Fitt}}

\def\M{{\mathcal M}}

\def\S{{\mathcal S}}

\def\Cu{{\mathcal C}}

\def\I{{\mathcal I}}

\def\reg{{\rm reg\,}}

\def\Hom{{\rm Hom }}
\def\codim{{\rm codim\,}}

\def\Im{{\rm Im\, }}

\def\Ker{{\rm Ker\,  }}

\def\Z{{\mathbb Z}}
\def\E{{\mathcal E}}

\def\Ok{{\mathcal O}}

\def\L{{\mathcal L}}

\def\Re{{\rm Re\,  }}

\def\L{{\mathcal L}}
\def\U{{\mathcal U}}

\def\J{{\mathcal J}}
\def\nbh{neighborhood }

\def\be{\begin{equation}}
\def\ee{\end{equation}}
\def\PM{{\mathcal{PM}}}

\def\BEF{\text{\tiny{bef}}}
\def\blow{{\nu}}

\newtheorem{thm}{Theorem}[section]
\newtheorem{lma}[thm]{Lemma}
\newtheorem{cor}[thm]{Corollary}
\newtheorem{prop}[thm]{Proposition}

\newtheorem*{thmA}{Theorem A}
\newtheorem*{thmB}{Theorem B}

\theoremstyle{definition}

\theoremstyle{remark}

\newtheorem{preremark}[thm]{Remark}
\newtheorem{preex}[thm]{Example}

\newenvironment{remark}{\begin{preremark}}{\qed\end{preremark}}
\newenvironment{ex}{\begin{preex}}{\qed\end{preex}}

\numberwithin{equation}{section}

\title[Global effective versions of the Brian{\c c}on-Skoda-Huneke
  theorem]{Global effective versions of the Brian{\c c}on-Skoda-Huneke
  theorem}

\begin{document}

\date{\today}

\author{Mats Andersson \& Elizabeth Wulcan}

\address{Department of Mathematics\\Chalmers University of Technology and the University of
Gothenburg\\S-412 96 G\"OTEBORG\\SWEDEN}

\email{matsa@chalmers.se \& wulcan@chalmers.se}


\thanks{The first author was
  partially supported by the Swedish
  Research Council. The second author was partially supported by
the Swedish  Research Council and by the NSF}

\begin{abstract} 

We prove global effective versions of the Brian\c con-Skoda-Huneke theorem.
Our results extend, to  singular varieties, a result of Hickel on the
membership problem in polynomial ideals in $\C^n$, and a related theorem of Ein and Lazarsfeld for smooth projective varieties.
The proofs rely on known geometric estimates and 
new results on multivariable residue calculus.

\end{abstract}

\maketitle

\section{Introduction}\label{puma}

Let $V$ be a reduced algebraic subvariety of $\C^N$ of pure dimension $n$.
If  $F_1, \ldots, F_m$ are  polynomials  in $\C^N$ with no common zeros on $V$, 
then by the  Nullstellensatz there are polynomials
$Q_j$ such that $\sum F_jQ_j=1$ on $V$. It was  proved by
Jelonek, \cite{Jel}, that if
$F_j$ have degree at most $d$, then  one can find  $Q_j$
such that
$$
\deg(F_jQ_j)\le c_m  d^\mu \deg V
$$
on $V$, where $c_m=1$ if $m\le n$, $c_m=2$ if $ m>n$, $\deg  V$ means
the degree of the closure of $V$ in $\P^N$, 
and,  throughout this paper,
$$
\mu:=\min(m,n).
$$
This result  generalizes Koll\'ar's theorem\footnote{In Koll\'ar's theorem $c_m=1$ 
even for $m>n$ (unless $d=2$, see also \cite{Sombra}) and this
estimate is optimal.}, \cite{Koll}, for $V=\C^n$ 
and does not require  any smoothness assumptions on  $V$.
The bound is optimal\footnote{In Koll\'ar's and
Jelonek's theorems, as well as in  \cite{Hick},  there are  more
precise results that take into account different degree
 bounds  $d_j$ of  $F_j$, but for simplicity, in this paper we always keep
all $d_j=d$.}
when $m\le n$ and almost optimal when $m>n$.
However, in view of various known results in the case when $V=\C^n$, one can expect
sharper degree estimates   if the common zero set of the polynomials $F_j$ 
behaves nicely at infinity in $\P^N$.

More generally one can take arbitrary polynomials $F_j$ 
of degree at most $d$ and look for a solution  $Q_j$ to
\begin{equation}\label{hummer}
F_1Q_1+\cdots +F_mQ_m=\Phi
\end{equation}
on $V$ with good degree estimates, provided that the polynomial $\Phi$ belongs to the
ideal $(F_j)$ generated by the $F_j$ on $V$. It follows from a result
of Hermann, \cite{Her}, that one can  choose $Q_j$ such that
$\deg(F_jQ_j)\le \deg \Phi+C(d,N)$, where $C(d,N)$ is like  $2(2d)^{2^N-1}$ for large $d$,
thus doubly exponential.
It is shown  in \cite{MM} that  this estimate cannot be substantially
improved for $V=\C^n$.
However, under additional hypotheses on $\Phi$ and 
the common zero
set of the $F_j$, much sharper estimates are possible.
In the extreme case when the polynomials $F_j$ have empty common zero set, even 
at infinity,
a  classical result of Macaulay, \cite{Macaul}, states that when
$V=\C^n$, one can solve \eqref{hummer}
with polynomials $Q_j$ such that 
$
\deg F_jQ_j\le \max(\deg \Phi, d(n+1)-n),
$
cf.\ Example~\ref{macex} below.  

\smallskip
By homogenization, this kind of effective results can be reformulated 
as geometric statements:  Let
 $z=(z_0,\ldots,z_N)$,  $z'=(z_1,\ldots,z_N)$, let
 $f_{i}(z):=z_0^{d}F_i(z'/z_0)$ be the $d$-homogenizations of $F_i$, and let
$\varphi(z):=z_0^{\deg \Phi}\Phi(z'/z_0)$.
Then there is a representation \eqref{hummer} on $V$ with $\deg(F_jQ_j)\le\rho$
if and only if there are  $(\rho-d)$-homogeneous forms $q_i$ on $\P^N$ such that
\begin{equation}\label{hoppla}
f_{1} q_1+\cdots +f_mq_m=z_0^{\rho-\deg\Phi}\varphi
\end{equation}
on the closure $X$ of $V$ in $\P^N$. As usual, we can consider  
$f_j$ as holomorphic  sections of 
(the restriction to
$X$ of) the line bundle $\Ok(d)\to\P^N$, $z_0^{\rho-\deg\Phi}\varphi$
as a section of
$\Ok(\rho)$, etc, so that \eqref{hoppla} becomes a statement about sections of
line bundles.

\smallskip

In this paper we present global effective versions of the  Brian\c con-Skoda-Huneke theorem:

\smallskip
{\it Let $\mathcal V$ be a germ of a reduced analytic set of pure
  dimension $n$ at 
  the origin in $\C^N$. There is a number $\mu_0$
such that if $a_1,\ldots,a_m,\phi$ are germs of holomorphic functions
at $0$,
$\ell\ge 1$,  and
$|\phi|\leq C |a|^{\mu+\mu_0+\ell-1}$ in a \nbh of $0$ in $\mathcal V$, where $C$ is a
positive constant and 
$
|a|^2=|a_1|^2+\cdots +|a_m|^2, 
$
then $\phi$ belongs to the ideal
$(a_1,\ldots,a_m)^\ell\subset \Ok_0$.}\footnote{
Often in the literature $\mu+\mu_0$ is replaced by a constant
independent of the number of generators ~$m$.}
\smallskip

If $\mathcal V$ is smooth, then one can take $\mu_0=0$; this is the
classical Brian\c con-Skoda theorem, \cite{BS}. The general case was proved by Huneke,
\cite{Hun},  by purely algebraic methods. An analytic proof appeared in \cite{ASS}.

\smallskip
Given polynomials  $F_1,\ldots, F_m$ on $V$, let $f_j$ denote the corresponding
sections of $\Ok(d)|_{X}$,  and let 
$\J_f$ be the coherent analytic sheaf on $X$ generated by $f_j$.
Furthermore, 
let $c_\infty$ be the maximal codimension of  the
so-called  {\it distinguished varieties} of the sheaf $\J_f$,
in the sense of Fulton-MacPherson,   
that are contained in 
$$
X_\infty:=X\setminus V,
$$
see Section~\ref{dist}.
If there  are no distinguished varieties contained in  $X_\infty$,
then we interpret $c_\infty$ as $-\infty$.
It is well-known that the codimension of a distinguished variety
cannot exceed the number $m$, see, e.g., Proposition~2.6 in \cite{EL}, and thus
\begin{equation}\label{bokslut}
c_\infty\le\mu.
\end{equation}
We let $Z^f$ denote the zero variety of $\J_f$ in $X$.

Our first result  involves the so-called {\it (Castelnuovo-Mumford) regularity}, $\reg X$,
of $X\subset\P^N$, see Section~\ref{lotta} for the definition.

\begin{thmA}
Assume that  $V$ is a reduced algebraic subvariety  of $\C^N$ of pure
dimension $n$ 
and let $X$ be its closure in $\P^N$. 

\smallskip
\noindent  (i)    There exists a number  $\mu_0$
such that if $F_1,\ldots,F_m$ are polynomials of degree
$\le d$ and $\Phi$ is a polynomial and
\begin{equation}\label{kaka}
|\Phi| /|F|^{\mu+\mu_0} \text{ is locally bounded on } V, 
\end{equation}
then there are polynomials $Q_1,\ldots, Q_m$ such that 
 \eqref{hummer} holds on $V$ and 
\begin{multline}\label{chimpans}
\deg (F_j Q_j)\leq 
\max \big(\deg \Phi + (\mu+\mu_0) d^{c_\infty} \deg X, (d-1)\min(m,n+1) +\reg X\big).
\end{multline}

\smallskip
\noindent (ii) If  $V$ is smooth, then there is a number $\mu'$ such that if
 $F_1,\ldots,F_m$ are polynomials of degree
$\le d$ and $\Phi$ is a polynomial and
\begin{equation}\label{kaka2}
|\Phi|/|F|^{\mu} \text{ is locally bounded on } V,
\end{equation}
 then there are polynomials $Q_1,\ldots, Q_m$ such that 
 \eqref{hummer} holds on $V$ and 
\begin{multline}\label{chimpans2}
\deg(F_j Q_j)\leq 
\max \big(\deg \Phi + \mu  d^{c_\infty}\deg X + \mu', (d-1)\min(m,n+1) +\reg X\big).
\end{multline}
If $X$ is smooth, then one can take $\mu'=0$.
\end{thmA}

\noindent 
There are  analogous results  for powers $(F_j)^\ell$  of $(F_j)$, see Theorem~\ref{borta}.

\smallskip
 
\noindent 
Note that if there  are no distinguished varieties of $\J_f$ contained in
$X_{\infty}$,  then $d^{c_\infty}=0$.

\smallskip
 

\noindent 
If 
$Z^f\cap X_{\text{sing}}=\emptyset$, 
then the conclusion in 
$(i)$ holds with $\mu_0=0$, see Remark~\ref{skola}.

\begin{remark} 
The number $\mu_0$ that appears in the proof of
Theorem ~A below only depends on the intrinsic variety $X$ and not on the
particular embedding $i:X\to \P^N$, cf.\ Remark~\ref{gbh}. 


\end{remark}

\begin{ex}
If we apply  Theorem~A to  Nullstellensatz  data,
i.e., $F_j$ with no common zeros on $V$
and $\Phi=1$,  we get back the
optimal result of Jelonek,  except for the
annoying factor $\mu+\mu_0$ in front of $d^{c_\infty}$. On the other hand,
$(\mu+\mu_0) d^{c_\infty}<d^\mu$ if $c_\infty<\mu$ and $d$ is large enough.
\end{ex}

\begin{ex}\label{macex} 
 If $f_j$ have no common zeros on $X$ (so that in particular
 $d^{c_\infty}=0$), then we can find a solution to 
 $F_1Q_1+\cdots +F_m Q_m=1$ on $V$ such that
$$
\deg F_jQ_j\le \max(\deg \Phi, (d-1)(n+1)+\reg X).
$$
If  $X=\P^n$, then $\reg X=1$ and hence we get back the
Macaulay theorem, cf., above. 
\end{ex}

\begin{remark} \label{CM}
Assume that $X\subset \P^N$ is Cohen-Macaulay; for instance, $X$ is a complete
intersection or even $X=\P^N$. 
Then $\reg X\le \deg X-(N-n)$,
see, \cite[Corollary 4.15]{Eis2}.
If in addition $m\le n$,  then the last entries in \eqref{chimpans} and \eqref{chimpans2}
can be omitted, i.e., we  get the sharper estimates 
$\deg (F_j Q_j)\leq \deg \Phi + (m+\mu_0) d^{c_\infty} \deg X$ and 
$\deg(F_j Q_j)\leq  \deg \Phi + m  d^{c_\infty}\deg X + \mu'$
in $(i)$ and $(ii)$, respectively, see the comment right after the
proof of Theorem ~A in Section ~\ref{helproofs}.

\end{remark}

\begin{remark}  If $X$ is smooth, then 
$$
\reg X\le (n+1)(\deg X-1)+1;
$$
this is Mumford's bound, see \cite[Example 1.8.48]{Laz}.
\end{remark}

\begin{ex}\label{uppstuds}
For $V=\C^n$,  Theorem~A gives the estimate
\begin{equation}\label{halvglad2}
\deg(F_j Q_j)\leq
\max \big(\deg \Phi + \mu d^{c_\infty}, d\min(m,n+1)-n\big). 
\end{equation}
This  estimate was proved by Hickel, \cite{Hick},
but with the term $\min(m,n+1)d^\mu$ rather than our
$\mu d^{c_\infty}$.  The ideas  in \cite{Hick} are similar
to the ones used in \cite{EL}.
If one applies  the geometric estimate
in \cite{EL},  rather than  the (closely related) so-called
refined Bezout estimate by Fulton-MacPherson 
that is used  in \cite{Hick}, one can replace the exponent 
$\mu$ by $c_\infty$. This refinement was pointed out already in
Example~1 in \cite{EL}.
\end{ex}

We have the following more abstract variant of Theorem~A.
It is a generalization to
nonsmooth varieties of the  geometric  effective Nullstellensatz
of   Ein-Lazarsfeld in  \cite{EL} (Theorem~\ref{elthm} below).
Let $X$ be a reduced projective variety.
Recall that if
 $L\to X$ is an {\it ample} line bundle, 
then there is a (smallest) number
$\nu_L$ such that $H^i(X,L^{\otimes s})=0$ for $i\ge 1$ and $s\ge \nu_L$,
cf., \cite[Ch.\ ~1.2]{Laz}.
When $X$ is smooth,
by Kodaira's vanishing theorem, $\nu_L$ is less than or equal to the least number $\sigma$
such that $L^\sigma\otimes K^{-1}_{X}$ is
strictly positive, where $K_X$ is the  canonical bundle.
In particular, if $V=\C^n$, i.e., $X=\P^n$,    then $\nu_{\Ok(1)}=-n$.

\begin{thmB}
Let $X$ be a reduced projective  variety of pure dimension $n$. There is a number $\mu_0$, only depending on $X$,
such that the following holds:
Let $f_1,\ldots,f_m$ be global holomorphic sections of an ample 
Hermitian line bundle $L\to X$,
and let $\phi$ be a section of  $L^{\otimes s},$
where  
\begin{equation}\label{nuv}
s\ge \nu_L +\min(m,n+1).
\end{equation}
If %
\begin{equation}\label{bs2}
|\phi| \leq C |f|^{\mu+\mu_0}, 
\end{equation}
then there are holomorphic sections $q_j$ of $L^{\otimes(s-1)}$ such that
\begin{equation}\label{likhet2}
f_1q_1+\cdots +f_mq_m=\phi.
\end{equation}
\end{thmB}

\noindent 
If $X$ is smooth we can choose $\mu_0=0$, see Theorem~\ref{elthm}.


\smallskip 

\noindent 
Let $\J_f$ be the ideal sheaf generated by $f_j$ and assume that
the associated distinguished varieties $Z_k$  have multiplicities $r_k$, cf., Section~\ref{dist}.
If we  assume that $\phi$ is in $\cap_k {\J(Z_k)}^{r_k(\mu+\mu_0)}$,
where $\J(Z_k)$ is the radical ideal associated with the distinguished variety $Z_k$,
then  \eqref{bs2} holds, and hence we have a representation \eqref{likhet2}.

\smallskip

\begin{ex}\label{orange}
Let $X$ be the cusp $\{z_1^2z_0^{p-2}-z_2^p=0\}\subset \P^2$, where
$p>2$ is odd. 
Then the sections $f=z_2$ of $L:=\Ok(1)|_X$ and
$\phi=z_0^{s-1}z_1$ of $L^{\otimes s}$ satisfy $|\phi|\leq C|f|^{\frac{p-1}{2}}$ on $X$ as soon as 
$s\geq 2$. 
However, $\phi$ is not in $(f)$ on $X$ at the singular point $\{z_1=z_2=0\}$ 
nor at $\{z_0=z_2=0\}$ (unless $p=3$).

One can check that $\gamma_p=(p-1)/2$ 
is the smallest
integer such that, for any choice of tuples $g_1,\ldots,
g_m$ of holomorphic germs at $\{z_1=z_2=0\}$, 
$|\psi|\leq C|g|^{1+\gamma_p}$
implies that the germ $\psi$ is in the local ideal $(g_j)$
at $\{z_1=z_2=0\}$, 
in other words  $\gamma_p$ is the
\emph{Brian\c con-Skoda number} at 
$\{z_1=z_2=0\}$.  
Moreover, one can check that  the Brian\c con-Skoda number at $\{z_0=z_2=0\}$
is $\lceil \frac{(p-3)(p-1)}{p-2}\rceil$, where $\lceil a \rceil$
denotes the smallest integer $\geq a$, 
see, e.g., \cite{Sz}. Therefore, $\mu_0$ must be at least 
\begin{equation*}\label{blubb} 
a_p:= \max \bigg (\frac{p-1}{2}, \Big \lceil
\frac{(p-3)(p-1)}{p-2}\Big \rceil\bigg ).
\end{equation*} 
In fact, in view of \cite[Section~2]{Sz} and the proof below one
finds that Theorem ~B holds with $\mu_0=a_p$ in this case. 


From \cite[Proposition ~4.16]{Eis2} we know that 
$\nu_L \leq \reg X -2 = p-2.$   
Since $\mu_0\geq a_p$, it follows that if our given  $f$ and $\phi$ satisfy \eqref{bs2}, then $s\geq p-1\geq
\nu_L+1$, so the
hypothesis \eqref{nuv} in Theorem ~B is vacuous in this case. 





\end{ex}

\noindent 
In particular, Example~\ref{orange} shows that $\mu_0$ can be arbitrarily
large.


\smallskip

The starting point for the proofs of Theorems~A
and~B is  the framework for solving division problems using residue
theory introduced in \cite{A3},
and further developed in \cite{AG, EW1, EW2}:
Assume that $X$ is a smooth
projective variety  and that $f_1,\ldots,f_m$ are sections of an ample line bundle $L\to
X$ with common zero set $Z^f$. From the
Koszul complex generated by the $f_j$
one defines  a current $R^f$ with support on $Z^f$ and taking
values in a direct sum of negative powers of $L$.
If  $\phi$ is a section of $L^{\otimes s}$ such that the current $R^f\phi$ vanishes,  and if in addition
$L^{\otimes s}$ is positive enough, so that certain cohomology classes
on $X$ vanish, and thus a certain sequence of $\dbar$-equations
can be solved  on $X$, one ends up with a holomorphic solution $q=(q_1,\ldots,q_m)$
to \eqref{likhet2}.

The main novelty in this paper is an extension of  this framework to
singular  $X$, see Section ~\ref{divandres}. Given an embedding $i\colon X\to Y$ of $X$ into a
smooth manifold   $Y$ and a locally free resolution on $Y$ of
$\Ok^Y/\J_X$, where $\J_X$ is the ideal sheaf associated with $X$, we construct an intrinsic  principal
value current $\omega$ on $X$, following \cite{AS}, 
and a ``product current''
$R^f\w\omega$.   If  $L$, $f_j$, and $\phi$  admit extensions to $Y$, which is
the situation  in Theorem~A with $Y=\P^N$,
we can proceed basically as before: If $R^f\w \omega \phi=0$, which is indeed an intrinsic condition on $X$,   
and certain cohomology classes  on $Y$ vanish, 
we end up with a holomorphic solution to
\eqref{likhet2};
this is how we prove Theorem~A.
By a small variation one can make this procedure more intrinsic and
assume vanishing of cohomology classes on $X$ rather than on $Y$. 

For the proof of Theorem ~B we cannot use this strategy directly,
since  we have no a~priori extensions of $L$, $f_j$, and $\phi$ to a smooth
manifold $Y$. However, for a
fixed $L$ it is possible to find an embedding $i\colon X\to Y$ such
that $L$ extends. 
Given such an embedding, without assuming  
holomorphic extensions of $f_j$ and $\phi$,  we construct a variant
$\widetilde R\wedge \omega$  on $X$ of $R^f\w\omega$, again with the property that if $\widetilde R\w \omega \phi=0$ and
the crucial cohomology classes vanish on $X$
we get a holomorphic solution to
\eqref{likhet2}.


To verify that the currents $R^f\wedge \omega\phi$
and $\widetilde R\w \omega\phi$ vanish 
we need to analyse the
singularities of $\omega$.
For Theorem~A it is enough to consider  a fixed $\omega$, coming from
the embedding of $X$ in $\P^N$ and a choice of resolution of
$\Ok^Y/{\J_X}$, whereas for Theorem~B we need to control the singularities
for $\omega$ coming
from any possible embedding $i:X\to Y$. To this end we need a new 
uniform estimate, Proposition~\ref{invarians}.


In Section~\ref{svart} we provide necessary
background on residue currents.
The proofs of our main theorems together with some  further results and comments
are gathered in Sections~\ref{dist} to~\ref{ormbo}.


\smallskip

\noindent
\textbf{Acknowledgement}:  
We would like to thank the referee for careful reading and many
valuable comments and suggestions that have substantially improved the exposition
of the paper.

\section{Some preliminaries on residue theory}\label{svart}

Most of the results in this section can be found in the papers 
\cite{A2, A12, AS, ASS, AW1, AW2}. More precisely, the
material in Section ~\ref{pm} is taken from 
\cite[Sections~2-3]{AW2}. Sections \ref{olga}-\ref{bef} are mostly based on
the first three sections in \cite{AW1} and Section ~\ref{lotta} is
based on Section ~6 in loc.\ cit. For Section ~\ref{tenprod}, see
\cite[Section~4]{A12}, and for Sections ~\ref{fundform} and
~\ref{fineres}, see Sections ~1 and ~3, respectively, in
\cite{AS}.
Finally Section \ref{boork} is based on   
\cite[Section~4]{ASS}. Proposition ~\ref{invarians} is new; the proof is given in Section
~\ref{pukas}.

Throughout this paper $X$ is a reduced projective variety of pure dimension ~$n$.
The sheaf
$\Cu_{\ell,k}$ of  currents of bidegree $(\ell,k)$ on $X$ is by definition the dual
of the sheaf  $\E_{n-\ell,n-k}$ of smooth  $(n-\ell,n-k)$-forms on $X$.
If $i\colon X\to Y$ is an embedding in a smooth manifold $Y$ of dimension $N$, then
$\E_{n-\ell,n-k}$ can be identified with the quotient sheaf
$\E_{n-\ell,n-k}^Y/\Ker i^*$, where $\Ker i^*$ is the sheaf of forms $\xi$ on $Y$ such that
$i^*\xi$ vanish on  $X_{\text{reg}}$. It follows that the currents $\tau$ in
$\Cu_{\ell,k}$ can be identified with
currents $\tau'=i_*\tau$ on $Y$ of bidegree $(N-n+\ell,N-n+k)$  that
vanish on $\Ker i^*$.

Given a holomorphic function $f$ on $X$, we have the principal value current
$[1/f]$, defined for instance as the limit
$$
\lim_{\epsilon\to 0}\chi(|f|^2/\epsilon)\frac{1}{f},
$$
where $\chi(t)$ is the characteristic function of the
interval $[1,\infty)$ or a smooth approximand of it.
The existence of this limit for a general $f$ relies on
Hironaka's theorem that ensures that there is a modification $\pi\colon\widetilde X\to X$
such that $\pi^*f$ is locally a monomial. It also follows that the
function $\lambda\to|f|^{2\lambda}(1/f)$, a~priori defined for
$\Re\lambda\gg 0$,  has
a current-valued analytic continuation to $\Re\lambda>-\epsilon$, and that
the value at $\lambda=0$ is precisely the current $[1/f]$, see, for instance,
\cite{Bj} or \cite{BjSam}.   
Although less natural at first sight, it turns out that this latter definition
via analytic continuation is often much more convenient. The same idea
will be used throughout this paper.
For the rest of this paper we skip the brackets and just write $1/f$.
It is readily checked that
\begin{equation}\label{kalkyl}
f\frac{1}{f}=1,\quad\quad f\dbar\frac{1}{f}=0.
\end{equation}

\subsection{Pseudomeromorphic currents}\label{pm}
In \cite{AW2} we introduced the sheaf $\PM$ of {\it pseudomeromorphic currents}
on $X$ in the case $X$ is smooth. The definition when $X$ is singular is identical.
In this paper we will use the slightly extended definition introduced in
\cite{AS}:
We say that a current of the form 
$$
\frac{\xi}{s_1^{\alpha_1}\cdots s_{n-1}^{\alpha_{n-1}}}\w\dbar\frac{1}{s_n^{\alpha_n}},
$$
where $s$ is a local coordinate system and $\xi$ is a smooth form with compact support,
is an {\it elementary pseudomeromorphic current}.
The sheaf $\PM$ consists of all possible (locally finite sums  of)
push-forwards under a sequence of maps
$ X^m\to\cdots \to X^1\to X$,
of elementary pseudomeromorphic currents, where $X^m$ is smooth, and each
mapping is either a modification, a simple projection
$\widehat X\times Y\to \widehat X$, or an open inclusion, i.e., $X^j$ is an
open subset of $X^{j-1}$. 

The sheaf $\PM$ is closed under $\dbar$ (and $\partial$) and multiplication by smooth forms.
If $\tau$ is in $\PM$ and has support on a subvariety $V$ and $\eta$ is
a holomorphic form that vanishes on $V$, then $\overline{\eta}\w\tau=0$.
We also have the

\noindent {\it Dimension principle:  If $\tau$ is a pseudomeromorphic current
on $X$ of bidegree $(*,p)$ that has support on a variety $V$ of codimension
$>p$, then $\tau=0$.}

\smallskip
If $\tau$ is in $\PM$ and $V$ is a subvariety of $X$, then the natural restriction
of $\tau$ to the open set $X\setminus V$ has a canonical extension
as a principal value to a pseudomeromorphic current  ${\bf 1}_{X\setminus V}\tau$ on $X$:
Let  $h$ be  a holomorphic tuple with common zero set
$V$. The current-valued function $\lambda\mapsto  |h|^{2\lambda}\tau$, a~priori defined
for $\Re\lambda\gg 0$,  has an analytic continuation to $\Re\lambda>-\epsilon$ and
its value at $\lambda=0$ is by definition  ${\bf 1}_{X\setminus V}\tau$.
One can also take a smooth approximand $\chi$ of the characteristic function
of  the interval $[1,\infty)$ and obtain ${\bf 1}_{X\setminus V}\tau$ as the limit
of $\chi(|h|^2/\epsilon)\tau$ when $\epsilon\to 0$.
It follows that  ${\bf 1}_{V}\tau:=\tau-{\bf 1}_{X\setminus V}\tau$ is
pseudomeromorphic and has  support on $V$. 
Notice that if $\alpha$ is a smooth form, then
${\bf 1}_{V}\alpha\w\tau=\alpha\w{\bf 1}_{V}\tau.$
Moreover, 
If $\pi\colon \widetilde X\to X$ is a modification,  $\tilde\tau$ is in $\PM(\widetilde X)$,
and
$\tau=\pi_*\tilde\tau$, then
\begin{equation}\label{enkelstjarna}
{\bf 1}_V\tau=\pi_*\big({\bf 1}_{\pi^{-1}V}\tilde\tau\big)
\end{equation}
for any subvariety $V\subset X$.
There is actually a reasonable
definition of ${\bf 1}_W\tau$ for any constructible set $W$,  and
\begin{equation}\label{dubbelstjarna}
{\bf 1}_W{\bf 1}_{W'}\tau={\bf 1}_{W\cap W'}\tau.
\end{equation}

Recall that a current is {\it semi-meromorphic}  if it is the quotient of a smooth
$L$-valued form and a holomorphic section of $L$, for some line bundle
$L$.
We say that a current $\tau$ is {\it almost semi-meromorphic}
in $X$ if there is a modification
$\pi\colon \widetilde X\to X$ and a  semi-meromorphic current
$\tilde \tau$  such that  $\tau=\pi_*\tilde\tau$, see
\cite[Section~2]{AS}. 
Analogously we say that $\tau$ is {\it almost smooth} if
$\tau=\pi_*\tilde\tau$ and $\tilde\tau$ is smooth.
Any almost semi-meromorphic (or smooth)
$\tau$ is pseudomeromorphic.

\subsection{Residues associated with Hermitian complexes}\label{olga}
Assume that
\begin{equation}\label{ecomplex}
0\to E_M\stackrel{f^M}{\longrightarrow}\ldots\stackrel{f^3}{\longrightarrow}
E_2\stackrel{f^2}{\longrightarrow}
E_1\stackrel{f^1}{\longrightarrow}E_0\to 0
\end{equation}
is a generically exact complex of Hermitian vector bundles over  $X$ and
let $Z$ be the subvariety  where \eqref{ecomplex}  is not pointwise exact.
The bundle $E=\oplus E_k$ gets  a natural superbundle structure, i.e., a $\Z_2$-grading,
$E=E^+\oplus E^-$, $E^+$ and $E^-$ being the subspaces of even and odd elements,
respectively,  by letting
$E^+=\oplus _{2k}E_k$ and $E^-=\oplus_{2k+1}E_k$. 
This extends to a $\Z_2$-grading of the sheaf $\Cu_\bullet(E)$ of $E$-valued currents,
so that  the degree of $\xi\otimes e$  is the sum of the current degree of $\xi$
and the degree of $e$, modulo $2$.
An endomorphism on $\Cu_\bullet(E)$ is even if it preserves degree and odd if it switches degrees.
The mappings $f:=\sum f^j$  and $\dbar$ are then odd mappings on 
$\Cu_\bullet(E)$. We introduce $\nabla=\nabla_f=f-\dbar$; it is
just (minus) the $(0,1)$-part of Quillen's superconnection $D-\dbar$.
Since the odd mappings $f$ and $\dbar$ anti-commute,  $\nabla^2=0$.
Moreover, $\nabla$  extends to an odd mapping  $\nabla_\End$ on $\Cu_\bullet(\End E)$ 
so that 
\begin{equation}\label{kalkyl2}
\nabla(\alpha\xi)=\nabla_\End\alpha\cdot\xi+(-1)^{\deg\alpha}\alpha\nabla\xi
\end{equation}
for sections $\xi$ and $\alpha$ of $E$ and $\End E$, respectively, 
and then  $\nabla_{\End}^2=0$.
In  $X\setminus Z$ we define, following \cite[Section~2]{AW1},
a smooth $\End E$-valued form $u$ such that
\begin{equation*}%
\nabla_{\End} u=I,
\end{equation*}
where $I=I_E$ is the identity endomorphism on  $E$. 
We have that 
$$
u=\sum_\ell u^\ell=\sum_{\ell}\sum_{k\ge \ell+1}u^\ell_k,
$$
where 
$u^{\ell}_{k}$ is in $\E_{0,k-\ell-1}(\Hom(E_\ell,E_{k}))$ over $X\setminus Z$.
Following \cite{AW1}\footnote{The definition is the  same  when $X$ is singular.}
we define  a pseudomeromorphic 
current extension $U$  of $u$ across  $Z$,
as the value at $\lambda=0$ of the current-valued analytic function
$$
\lambda\mapsto U^\lambda:=|F|^{2\lambda} u,
$$
a~priori defined for $\Re\lambda\gg 0$, 
where $F$ is the tuple $f^1$.  
In the same way we define the {\it residue current} $R$   associated with
\eqref{ecomplex} as the value at $\lambda=0$ of
$$
\lambda\mapsto R^\lambda:= (1-|F|^{2\lambda})I +\dbar|F|^{2\lambda}\w u.
$$
The existence of the analytic continuations follows
from a suitable resolution $\widetilde X\to X$, see \cite{AW1}, see
also Section~\ref{dist} below. 
The current $R$ clearly has support on  $Z$, and 
$$
R=\sum_\ell R^\ell=\sum_\ell\sum_{k\ge\ell+1} R^\ell_k,
$$
where   $R^\ell_k$ is a $\Hom(E_\ell,E_k)$-valued $(0,k-\ell)$-current. 
The currents $U^\ell$ and $U^\ell_k$ are defined analogously. 
Notice that $U$ has odd degree and that $R$ has even degree.
By the dimension principle, 
$R^\ell_{k}$ vanishes if $k-\ell<\codim Z$.
In particular, $R^0_0=(1-|F|^{2\lambda})I_{E_0}|_{\lambda=0}$ is  zero, unless some
components $W$ of
$Z$ has codimension $0$,  in which  case $R^0_0$ is the characteristic
function for $W$ times the identity $I_{E_0}$ on $E_0$.
However, when we define products of currents later on, all components
of $R^\lambda$ may  play a role.

Since $\nabla_\End U^\lambda=I-R^\lambda$ and $\nabla_\End R^\lambda=0$
when $\Re\lambda\gg 0$, we conclude that
\begin{equation}\label{urformel}
\nabla_{\End}U= I-R, \qquad  \nabla_{\End} R=0. 
\end{equation}
In particular, if $\xi$ is a section of $E$, then 
\[
\nabla(U\xi)=\xi-R\w\xi. 
\]
Also, \eqref{urformel} means that, cf. \eqref{kalkyl2},  
\begin{equation*}%
f^1 U_1^0=I_{E_0}, \quad  f^{k+1}U_{k+1}^0-\dbar U_k^0=R_k^0; \quad  k\ge 1.
\end{equation*}
Notice that when $\phi$ is a section of $E_0$, then
$R^0\phi=R\phi$ and $U^0\phi=U\phi$, and we will often
skip the upper indices. 

\begin{ex}[The Koszul complex]\label{koszulex}
Let $f_1,\ldots, f_m$ be holomorphic sections of a Hermitian line bundle $L\to X$.
Let $E^j$ be disjoint trivial line bundles with basis elements $e_j$ and
define the rank $m$ bundle
$$
E=(L^{-1}\otimes E^1) \oplus\cdots \oplus (L^{-1}\otimes E^m)
$$
over $X$.
Then $f:=\sum f_j e_j^*$, where $e_j^*$ is the dual basis, is a
section of the dual bundle $E^*= L\otimes (E^1)^* \oplus\cdots \oplus L\otimes (E^m)^*$.
If $S\to X$ is a Hermitian line bundle we can form  a complex \eqref{ecomplex} with
$$
E_0=S, \quad     E_k=S\otimes \Lambda^k E,
$$
where  all the mappings $f^k$ in \eqref{ecomplex}
 are  interior multiplication $\delta_f$ by the section $f$.
Notice that
\begin{equation*}%
E_k=S\otimes L^{-k}\otimes \Lambda^k(E^1\oplus\cdots \oplus E^m).
\end{equation*}
The superstructure of $\oplus_k E_k$ in this case coincides with the natural grading of
the exterior algebra $\Lambda E$ of $E$ modulo $2$.

Let us  recall how  the currents $U^0$ and $R^0$ are defined in this case.
For simplicity we suppress the upper indices throughout this example.
We have the natural norm
$$
|f|^2=\sum_j |f_j|^2_L
$$
on $E^*$. Let $\sigma$ be  the section of $E$ over $X\setminus Z$
of pointwise minimal norm such that $f\cdot\sigma=\delta_f\sigma=1$, i.e.,
\begin{equation}\label{farfar}
\sigma = \sum_j \frac {f_j^* e_j}{|f|^2},
\end{equation} 
where $f_j^*$ are  the sections of $L^{-1}$ of minimal norm such that $f_jf_j^*=|f_j|^2_L$.

Let us consider the exterior algebra over $E\oplus T^*(X)$ so that
$d \bar z_j\w e_\ell =-e_\ell\w d\bar z_j$ etc. 
Then,  e.g., $\dbar\sigma$ is a form of positive
degree.
We have  the smooth form 
\begin{equation}\label{udefined}
u=\sum u_k, \quad  u_k=\sigma\w(\dbar\sigma)^{k-1} 
\end{equation}  
in $X\setminus Z$, and it admits
a natural current extension $U$ across $Z$, e.g., defined as the analytic continuation
of
$
U^{\lambda}=|f|^{2\lambda}u
$
to  $\lambda=0$.
Furthermore, the associated residue current
$R$ is obtained as the evaluation at $\lambda=0$ of
\begin{multline*}
R^{\lambda}:=1-|f|^{2\lambda}+
\dbar|f|^{2\lambda}\w u =\\
1-|f|^{2\lambda}+\dbar|f|^{2\lambda}\w u_1+\cdots + \dbar|f|^{2\lambda}\w u_{\min(m,n)}
=:R^{\lambda}_0+R^{\lambda}_1+\cdots +R^{\lambda}_{\min(m,n)}.
\end{multline*}
The current $R$ 
was introduced in \cite{A2} in this form, much  inspired by \cite{PTY} where
the coefficients appeared.
\end{ex}

\subsection{The associated sheaf complex}\label{sheafcomplex}

Given the complex \eqref{ecomplex} we
have the  associated  complex of locally free sheaves
\begin{equation}\label{sheaves}
0\to \mathcal{O}(E_M)\stackrel{f^M}{\longrightarrow}\ldots\stackrel{f^3}{\longrightarrow}
\mathcal{O}(E_2)\stackrel{f^2}{\longrightarrow}
\mathcal{O}(E_1)\stackrel{f^1}{\longrightarrow}\mathcal{O}(E_0).
\end{equation}
In this paper $E_0$ is always a line bundle so that $\J:=\Im f^1$ is a
coherent ideal sheaf over  $X$.

Consider the double sheaf complex $\M_{\ell, k}:=\Cu_{0,k}(E_\ell)$ with mappings
$f$ and $\dbar$. We have the associated total complex
$$
\ldots \stackrel{\nabla_f}{\longrightarrow} \M_j \stackrel{\nabla_f}{\longrightarrow}
\M_{j-1}\stackrel{\nabla_f}{\longrightarrow} \ldots\quad ,
$$
where 
$
\M_j=\oplus_{\ell-k=j}\M_{\ell,k}.
$ 
If $X$ is smooth, then  $\M_{\ell,k}$  is exact in the $k$-direction except at $k=0$,
and the kernels there are $\Ok(E_\ell)$.  Notice that if
$\phi$ is in $\Ok(E_\ell)$ and $f^\ell\phi=0$, then also $\nabla_f\phi=0$. We therefore
have a natural mapping
\begin{equation}\label{bus1}
 H^j(\Ok(E_\bullet))\to   H^j(\M_\bullet).
\end{equation}
By   standard homological algebra, \eqref{bus1} is in fact
an isomorphism. We can also consider the corresponding sheaf complexes
$\M^\E_{\ell,k}:= \E_{0,k}(E_\ell)$,
$\M^\E_j=\oplus_{\ell-k=j}\M^\E_{\ell,k}$ of smooth sections, and  the  analogue of
\eqref{bus1} is then an  isomorphism as well.

\begin{lma}\label{flundra}
Assume that $X$ is smooth. If $\phi$ is a holomorphic section of $E_0$ that annihilates $R$, i.e.,
$R\phi=0$, then  $\phi$ is in $\J$.
\end{lma}

\begin{proof}
In fact,  by \eqref{urformel} we have that
\begin{equation*}%
\nabla_f (U\phi)=\phi- R\phi=\phi.
\end{equation*}
Since $X$  is smooth, \eqref{bus1} is an isomorphism, and thus 
locally $\phi =f^1 \psi$ for some holomorphic $\psi$,
i.e., $\phi$ is in $\J$.
\end{proof}

The smoothness assumption is crucial, as the following example shows.

\begin{ex} Let  $f$ be one single function.
Then the residue condition $R\phi=0$  means that $\dbar (\phi/f)=0$.
Thus $\psi=\phi/f$ is   in the Barlet-Henkin-Passare
class, cf., \cite{HP} and \cite{AS}; however in general $\psi$ is not
(strongly) holomorphic, i.e., in general  $\phi$ is not in $\J=(f)$.
\end{ex}

We shall now see that if $X$ is smooth and  there is a global current solution to $\nabla W=\phi$, then
there is also a global smooth solution.
For further reference however we need  a slightly more general statement about  the associated complex  of
global sections.
Let  $\M_{\ell,k}(X)$ and $\M^\E_{\ell,k}(X)$ be the double  complexes 
of
global current valued  and smooth sections,  respectively, and let $\M_\bullet(X)$
and $\M^\E_\bullet(X)$ be the associated total complexes.
Notice that we have  natural mappings
\begin{equation}\label{puma1}
H^j(\M^\E_\bullet(X))\to H^j(\M_\bullet(X)),  \quad j\in \Z.
\end{equation}
The following result is standard, but we include a proof for the
reader's convenience.

\begin{prop}\label{urlaka}
If  $X$ is smooth, then the mappings  \eqref{puma1}
are  isomorphisms.
\end{prop}

\begin{proof}
By the de~Rham theorem, the natural mappings
\begin{equation}\label{puma2}
H^k(\E_{0,\bullet}(X,E_\ell))\to H^k(\Cu_{0,\bullet}(X,E_\ell)),  \quad k\in \Z,
\end{equation}
are isomorphisms; these spaces are in fact naturally isomorphic to
the cohomology groups $H^k(X,\Ok(E_\ell))$.
The short exact sequence
$$
0\to\M_\bullet^\E(X)\to \M_\bullet(X)\to\M_\bullet(X)/\M_\bullet^\E(X)\to 0
$$
gives rise to, for each fixed $\ell$,   the long exact sequence
\begin{multline*}
\ldots \to H^{k-1}(\E_{0,\bullet}(X,E_\ell))\to   H^{k-1}(\Cu_{0,\bullet}(X,E_\ell))\to \\
 H^{k-1}(\Cu_{0,\bullet}(X,E_\ell)/\E_{0,\bullet}(X,E_\ell))\to H^{k}(\E_{0,\bullet}(X,E_\ell))\to \ldots \quad,
\end{multline*}
and since \eqref{puma2} are isomorphisms    
the cohomology in the $k$-direction of $\M_{\ell,k}(X)/\M^\E_{\ell,k}(X)$
is zero.  By a simple homological algebra argument,
using that the double complexes involved are bounded, 
it follows that 
$$
H^{j}(\M_\bullet(X)/\M^\E_\bullet(X))=0
$$ 
for each $j$.
The proposition now follows from the long exact sequence
\begin{multline*}
\ldots \to H^{j-1}(\M^\E_\bullet(X))\to  H^{j-1}(\M_\bullet(X))\to \\
 H^{j-1}(\M_\bullet(X)/\M^\E_\bullet(X)) \to H^{j}(\M^\E_\bullet(X))\to \ldots \quad.
\end{multline*}
\end{proof}

\subsection{BEF-varieties and duality principle}\label{bef}

We now consider the case when the locally free complex \eqref{sheaves}
is exact, i.e., a resolution of the sheaf  $\Ok(E_0)/\J$. 
We will refer to a (locally free) resolution $\Ok(E_0)/\J$ together with a 
choice of Hermitian metrics on the corresponding vector bundles $E_k$ as a \emph{ Hermitian 
(locally free) resolution}. 
Let $Z_k^\BEF$ be the set where the mapping $f^k$ does not
have optimal rank. 
Then
\[
\cdots Z^\BEF_{k+1}\subset Z^\BEF_k\subset \cdots\subset Z^\BEF_1=Z,  
\] 
and these sets 
 are independent of the choice of resolution;  
we call them the   {\it BEF varieties}\footnote{The sets  $Z_k^\BEF$ are  the zero varieties
of  certain Fitting ideals associated with a free resolution of $\Ok^X/\J$;
the importance of these sets (ideals) was pointed out by Buchsbaum and Eisenbud
in the 70's. We have not seen any notion for these sets in the literature,
and  ``Buchsbaum-Eisenbud varieties'' is already occupied for another
purpose,  so we stick to BEF as an acronym for
Buchsbaum-Eisenbud-Fitting.}.
It follows from the Buchsbaum-Eisenbud theorem   that
$\codim Z_k^\BEF\ge k$.
If moreover $\J$ has pure dimension, for instance
$\J$ is the radical ideal sheaf of a pure-dimensional subvariety, then
$\codim Z_k^\BEF\ge k+1$ for $k \ge 1+\codim\J$, see \cite[Corollary~20.14]{Eis}.

Since \eqref{sheaves} is exact, by  \cite[Theorem~3.1]{AW1}, we have that
$R^\ell=0$ for each $\ell\ge 1$, i.e.,  $R=R^0$. 
Moreover, there are almost semi-meromorphic
$\Hom(E_k, E_{k+1})$-valued $(0,1)$-forms $\alpha_{k+1}$,
that are smooth outside $Z^\BEF_{k+1}$, such that
$$
R_{k+1}=\alpha_{k+1} R_k
$$
there, see \cite[Section~3]{AW1}. From  \cite[Theorem~1.1]{AW1} we also have the

\smallskip

\noindent{\it Duality principle:  If $X$ is smooth and
\eqref{sheaves}   is  a resolution of the sheaf
$\Ok(E_0)/\J$, then $\phi\in\J$ if and only if
$R\phi=0$.}

\smallskip
\noindent That is, the annihilator ideal sheaf of the residue current $R$ is precisely
the ideal sheaf $\J$ generated by $f^1$.

If for instance $f^1=(f_{1},\ldots,f_{m})$ defines a complete
intersection, i.e, $\codim Z= m$, 
then the Koszul complex is a resolution of
$\J$ and hence the duality principle states that the annihilator
of the residue current  in Example~\ref{koszulex} is the ideal
itself. 

\subsection{Tensor products of complexes}\label{tenprod}

Assume  that $(E^g_\bullet, g)$ and $(E^h_\bullet, h)$ are Hermitian complexes.
We can then  define a   complex  $(E^f_\bullet= E^g_\bullet\otimes E^h_\bullet, f)$, where
\begin{equation*}%
E_k^f=\bigoplus_{i+j=k} E^g_i\otimes E^h_j,
\end{equation*}
and $f=g+h$, or more formally
$f=g\otimes I_{E^h}+ I_{E^g}\otimes h$,
such that
\begin{equation}\label{krux}
f(\xi\otimes\eta)=g\xi\otimes\eta +(-1)^{\deg\xi}\xi\otimes h\eta.
\end{equation}
Notice that $E^f_0$ is the line bundle $E_0^g\otimes E_0^h$. If $g^1\Ok(E_1^g)=\J_g$ and $h^1\Ok(E^h_1)=\J_h$,  then
$f^1\Ok(E^f_1)=\J_g+\J_h$.
One extends \eqref{krux} to current-valued sections $\xi$ and $\eta$,  and
$\deg\xi$ then means  total degree.
We write $\xi\cdot\eta$, or sometimes $\xi\w\eta$ to emphasize
that the sections may be form- or current-valued, rather than $\xi\otimes\eta$, and
define
\begin{equation}\label{opera}
 \eta\cdot\xi=(-1)^{\deg \xi \deg\eta}\xi\cdot\eta.
\end{equation}
Notice that
\begin{equation*}%
\nabla_f (\xi\cdot\eta)=\nabla_g\xi\cdot\eta
+(-1)^{deg \xi}\xi\cdot \nabla_h\eta.
\end{equation*}
Let  $u^g$ and $u^h$ be the corresponding
$\End(E^g)$-valued and $\End(E^h)$-valued
forms, cf., Section~\ref{olga}.
Then $u^h\w u^g$ 
is a $\End(E^f)$-valued form
defined outside $Z^g\cup Z^h$.
Following the proof of  Proposition~2.1 in \cite{AW2} we can define
$\End(E^f)$-valued pseudomeromorphic currents
$$
U^h\w R^g:=U^{h,\lambda}\w  R^g |_{\lambda=0}, \quad
R^h\w  R^g:=R^{h,\lambda}\w  R^g |_{\lambda=0}.
$$
We have that, cf., \eqref{urformel} and \cite[Section~4]{A12},
$$
\nabla_{\End,f}(U^h\w R^g +U^g)=I_{E^f}-R^h\w R^g.
$$
In general, the current  $R^h\w R^g$ will change if we
interchange the roles of $g$ and $h$.

In particular we can form the product  $E^h_\bullet\otimes E^h_\bullet$
of $E^h_\bullet$ by itself. In this case we consider \eqref{opera}
as an identification, so that, for instance,
$$
(E^h_\bullet\otimes E^h_\bullet)_1=E^h_1\dot{\otimes} E_0^h, \quad
(E^h_\bullet\otimes E^h_\bullet)_2=E^h_2\dot{\otimes} E_0^h +\Lambda^2 E^h_1,
$$
etc, where $\dot{\otimes}$  denotes symmetric tensor product.
In general, $\xi\cdot\xi=0$ if  $\xi$ has odd (total) degree.

We can just as well form a similar product of more than two complexes, and in particular,
we can form the  product $(E^h)^{\otimes k}=
E^h\otimes E^h\otimes \cdots \otimes E^h$
of a given complex by itself.

\subsection{The   structure form $\omega$ on a singular variety}\label{fundform}
Let  $i\colon X\to Y$ be an embedding of $X$ in a  smooth projective
manifold $Y$ of dimension $N$,
let $\J_X$ be the radical ideal sheaf associated with $X$ in $Y$,
and let $S\to Y$ be an ample Hermitian line bundle.
Moreover, let $E_{k}^j$ be disjoint trivial line bundles over $Y$
with basis elements $e_{k,j}$. There is a
(possibly infinite) resolution, see, e.g., \cite[Ch.1, Example~1.2.21]{Laz},
\begin{equation}\label{gcomplex}
\ldots \stackrel{g^3}{\longrightarrow}\Ok(E_2)\stackrel{g^2}{\longrightarrow}\Ok(E_1)
\stackrel{g^1}{\longrightarrow}\Ok(E_0)
\end{equation}
of $\Ok(E_0)/\J_X=\Ok^X$, where $E_k$ is of the form 
$$
E_k=\big(E^1_k\otimes S^{-d^1_k}\big)\oplus\cdots \oplus
\big(E^{r_k}_k\otimes S^{-d_k^{r_k}}\big),\quad E_0=E_0^0\simeq\C,
$$
$E_k^i$ are trivial line bundles, and
$$
g^k=\sum_{ij} g^k_{ij}e_{k-1,i}\otimes e_{k,j}^*
$$
are matrices of sections
$$
g^k_{ij}\in\Ok(Y,S^{d_k^j-d_{k-1}^i});
$$
here $e_{k,j}^*$ are the dual basis elements.
There are natural induced norms on  $E_k$. 
The associated residue current\footnote{The fact that \eqref{gcomplex}
may be  infinite  causes no problem, since,  for degree
reasons,  $U$ and $R$ only contain a finite number of terms.}
$R$ is annihilated by all
smooth forms $\xi$ such that $i^*\xi=0$.
Let $\varOmega$ be
a  global nonvanishing $(\dim Y,0)$-form with values
in $K_Y^{-1}$.
By  \cite[Proposition~16]{AS}    there is a (unique) almost
semi-meromorphic   current
$\omega$ on  $X$, smooth on $X_{\text{reg}}$,  such that
\begin{equation}\label{poseidon}
i_*\omega= R\w\varOmega.
\end{equation}
We say that $\omega$ is a {\it structure form} on $X$.

As an immediate consequence of the existence of $\omega$, the product $\alpha\w R$ is well-defined 
for all (sufficiently) smooth forms $\alpha$ on $X$. If $\alpha=i^* a$, we let $\alpha\w R:=a\w R$. 
This product only depends on $\alpha$, since if $i^*a=0$, then $a\w R\w \varOmega = i_*(i^* a\w \omega)=0$ 
and hence $a\w R=0$ since $\varOmega \neq 0$. 

\smallskip
Let $X_k$ be the BEF varieties of $\J_X$, and define
\begin{equation*}
X^0=X_{\text{sing}}, \quad X^\ell=X_{N-n+\ell},  \  \ell\ge 1.
\end{equation*}
Since $\J_X$ has pure dimension it follows that 
\begin{equation}\label{nian}
\codim X^k\ge k+1,
\end{equation}
and  in particular,
$X^n=\emptyset$. These sets $X^\ell$ are actually  independent of the choice of embedding
of $X$, cf., the comment  after Lemma~\ref{purr1}.

Let $g_\ell$ be the restriction to $X$ of $g^{N-n+\ell}$,
and let $\nabla^g=g-\dbar$ on $X$. Let $E^\ell=E_{N-n+\ell}|_X$.
Then
$
\omega=\omega_0+\omega_1+\cdots +\omega_n,
$
where $\omega_\ell$ is a $(n,\ell)$-form on $X$ taking values in $E^{\ell}$,
and
$
\nabla^g \omega =0
$
on $X$.
There are almost semi-meromorphic
$\Hom(E^\ell, E^{\ell+1})$-valued $(0,1)$-forms  $\alpha^{\ell+1}$
such that
\begin{equation}\label{snorlik}
\omega_{\ell+1}=\alpha^{\ell+1} \omega_\ell
\end{equation}
there.
In fact, $\alpha^{\ell}$ is the pullback to  $X$ of the form
$\alpha_{N-n+\ell}$ associated with a resolution of $\Ok^Y/\J_X$ in $Y$,
cf., Section~\ref{bef}.

Since $\omega$ is almost semi-meromorphic, it has the {\it the standard extension
property}, SEP on $X$, which means that ${\bf 1}_W \omega=0$ for all varieties $W\subset X$ of positive codimension.

The singularities of a structure form  $\omega$ only depend on  $X$,  in the following sense:

\begin{prop}\label{invarians}
Let $X$ be a projective  variety. There is a smooth modification
$\tau\colon\widetilde X\to X$ and a holomorphic section $\eta$
of a line bundle $S\to \widetilde X$ such that the following holds:
If  $X\to Y$ is an embedding of $X$ in a smooth manifold $Y$,
$\big (\Ok(E^g_\bullet),g\big )$ is a Hermitian locally free resolution  of $\Ok^Y/\J_X$, and $\omega$ is the
associated structure  form on $X$, then $\eta\tau^*\omega$ is smooth
on $\widetilde X$. We can choose $\eta$ to be nonvanishing in
$\widetilde X\setminus \tau^{-1} X_{\text{sing}}$. 
\end{prop}
After further resolving we may assume that $\eta$ is locally a monomial
in $\widetilde X$.

The proof is postponed to Section~\ref{pukas}.
Since $\omega$ is almost semi-meromorphic, the pullback $\tau^*\omega$ is well-defined;
this  follows from the proof below, cf., also
the remark after Definition~12 in \cite{AS}.

\subsection{Local division problems on a singular variety}\label{boork}
Still assume that we have the embedding $i\colon X\to Y$, where $Y$ is smooth,
and the complex $(E_\bullet^g, g)$  over $Y$ 
corresponding to a Hermitian locally free resolution of $\Ok^Y/\J_X$.
If   $(E^f_\bullet,f)$ is an arbitrary Hermitian complex over $Y$
we have  the complex $E^F=E^f\otimes E^g$ with mappings
$F=f+g$ as in Section~\ref{tenprod}.
Let $F^k=F|_{E_k}$. Since
$
R^f\w R^g=R^{f,\lambda}\w R^g|_{\lambda=0}$ and  $U^f\w R^g=U^{f,\lambda}\w R^g|_{\lambda=0}
$, 
cf., Section~\ref{fundform},
these currents only depend on the values of $f$ on $X$. From Section~\ref{tenprod}
we also have that
\begin{equation}\label{badanka}
\nabla_{\End, F} U=I-R^f\w R^g
\end{equation}
if 
$
U=U^f\w R^g+U^g.
$
If $\Phi$ is a (locally defined)  holomorphic function in $Y$ and
$R^f\w R^g\Phi=0$, then,  following the proof of Lemma~\ref{flundra}, 
there is a local holomorphic solution $v=v_f+v_g$ in $E_1^F=E^f_1\otimes E^g_0+E^f_0\otimes E^g_1$
to $g^1 v_f +f^1 v_g=F^1 v=\Phi$.
Notice that in fact $R^f\w R^g\Phi$ only depends on the class $\phi$ of $\Phi$
in $\Ok^Y/\J_X=\Ok^X$, so $R^f\w R^g\phi$ is well-defined for
$\phi$ in $\Ok^X$.
We can define the intrinsic  residue current
$$
R^f\w\omega := R^{f,\lambda}\w\omega|_{\lambda=0}
$$
on $X$. Since  $i_*R^{f,\lambda}\w\omega=R^{f,\lambda}\w R^g\w\varOmega$
when  $\Re\lambda\gg 0$, we conclude that
$$
i_*R^f\w\omega=R^f\w R^g\w\varOmega.
$$
Since $\varOmega$ is nonvanishing,
$R^f\w\omega \phi=0$ implies that $R^f\w R^g\phi=0$ and thus we have:

\begin{prop}
 Assume that $(E^f_\bullet, f)$ is a Hermitian complex on $X$.
If $\phi$ is a holomorphic section of $E_0^f$  on $X$ such that $R^f\w \omega\phi=0$,
then locally $\phi$ is in the image of $f^1$ on $X$.
\end{prop}

\subsection{A fine resolution of $\Ok$ on $X$}\label{fineres}

It was proved in  \cite{AS}, see \cite[Theorem~2]{AS}, that   
there exist  sheaves $\A_k$ of $(0,k)$-currents on $X$ with the following properties:

\noindent $(i)$ $\A_k$ is equal to $\E_{0,k}$ on $X_{\text{reg}}$,

\noindent $(ii)$ $\A=\oplus_k\A_k$ is closed under multiplication by smooth $(0,*)$-forms,

\noindent $(iii)$ $\dbar$ maps $\A_k$ to $\A_{k+1}$ and if $E$ is any vector bundle over $X$,
then the sheaf complex
\begin{equation*}
0\to \Ok(E)\to\A_0(E)\stackrel{\dbar}{\longrightarrow}\A_1(E)
\stackrel{\dbar}{\longrightarrow}\A_2(E) \stackrel{\dbar}{\longrightarrow}\ldots
\end{equation*}
is exact.

\smallskip
\noindent By standard sheaf theory we have canonical isomorphisms
$$
H^k(X,\Ok(E))=\frac{\Ker \big(\Gamma(X,\A_k(E))\stackrel{\dbar}{\to}\Gamma(X,\A_{k+1}(E))\big)}
{\Im  \big(\Gamma(X,\A_{k-1}(E))\stackrel{\dbar}{\to}\Gamma(X,\A_k(E))\big)},
\quad\quad k\ge 1.
$$

\subsection{Subvarieties of $\P^N$}\label{lotta}

Let $X$ be a subvariety of $Y=\P^N$, $S=\Ok(1)$, and let
$(\Ok(E_\bullet),g)$ be a resolution of $\Ok(E_0)/\J_X$ as in \eqref{gcomplex}.
Then, see \cite[Section~6]{AW1}, 
\begin{equation*}
E_k=\big(E^1_k\otimes \mathcal{O}(-d^1_k)\big)\oplus\cdots \oplus
\big(E^{r_k}_k\otimes \mathcal{O}(-d_k^{r_k})\big)
\end{equation*}
and $g^k=(g^k_{ij})$
are matrices of homogeneous forms with
$
\deg g^k_{ij}= d_k^j-d_{k-1}^i .
$
We choose the Hermitian metrics so that 
$$
|\xi(z)|^2_{E_k}=\sum_{j=1}^{r_k}|\xi_j(z)|^2 |z|^{2d^j_k}
$$
if  $\xi=(\xi_1,\ldots,\xi_{r_k})$ is a section of $E_k$.
Moreover,
$$
\varOmega=const\times \sum (-1)^jz_j dz_0\w\ldots \w\widehat{dz_j}\w\ldots\w dz_N
$$
in $\P^N$.

Let $J_X$ denote the homogeneous ideal in the graded ring $\S=\C[z_0,\ldots,z_N]$
that is associated with  $X$, and let $\S(\ell)$ denote the module $\S$ where
all degrees are shifted by $\ell$.
Then $(\Ok(E_\bullet),g)$   corresponds to a free resolution
\begin{equation}\label{svenne}
\ldots \to\oplus_i \S(-d_k^i)\to\ldots  \to\oplus_i \S(-d_2^i)\to\oplus_i \S(-d_1^i)\to \S
\end{equation}
of the module $\S/J_X$. Conversely, any such free resolution corresponds to a sheaf resolution
$(\Ok(E_\bullet),g)$.

Notice that the ideal $J_X$ has pure dimension in $\S$, 
so that in particular the ideal associated to the origin is not an associated prime ideal.
From Corollary~20.14 in \cite{Eis},
applied to $\S$, it follows that the BEF-variety of dimension zero
must vanish;  therefore
the depth of $\S/J_X$ is at least $1$, and hence a minimal
free resolution of $\S/J_X$ has length $\le N$.
Recall that the  {\it (Castelnuovo-Mumford) regularity} of a homogeneous module with free graded resolution
\eqref{svenne} is defined as $\max_{k,i}(d_k^i-k)$,
see, e.g., \cite[Ch.~4]{Eis2}. The regularity $\reg X$  of $X\subset\P^N$ is defined
as the regularity of the ideal $J_X$, which is, cf., \cite[Exercise~4.3]{Eis2}, equal to 
$\reg (\S/J_X) +1$; note that $\reg X$  depends on the embedding
of $X$ in $\P^N$. If the minimal free resolution of $\S/J_X$ has length $M\le N$ we conclude that
\begin{equation}\label{spasm}
\reg X = \max_{k\le M} (d_k^i-k) +1.
\end{equation}
The regularity of $X$ is also equal to the {\it (Castelnuovo-Mumford) regularity} of the sheaf $\I_X$, 
see again  \cite[Exercise~4.3]{Eis2}.

\section{Singularities of the structure form}\label{pukas}

In this section we provide a proof of Proposition~\ref{invarians}.
Let  $i\colon X\to Y$ be an embedding where $Y$ is projective and smooth
of dimension $N$. %
Recall that the $k$th \emph{Fitting ideal (sheaf)} of $\Ok^Y/\J_X$, 
$\Fitt_0 g^k$, is the ideal  generated by the $r_k$-minors
of (the matrix) $g^k$ in a locally free resolution
$\big(\Ok(E^g_\bullet),g\big )$ of $\Ok^Y/\J_X$, 
where $r_k$ is the generic rank of $g^k$, see, e.g., \cite{Eis}.
It is well-known that these ideals   are independent of the 
resolution $\big (\Ok(E^g_\bullet),g\big )$; the zero variety of $\Fitt_0 g^k$ is just the BEF-variety $Z_k^\BEF$, 
cf., Section~\ref{bef}. 
Since $X$ has pure dimension, $\Fitt_0 g^k$ is trivial
when $k\ge N$, cf., \eqref{nian}. Let $p=N-n$ be the codimension of $X$ in $Y$. For 
$\ell=1,\ldots,n-1$, let 
$\a_\ell$ be the pullback (restriction)  of  $\Fitt_0 g^{p+\ell}$ to $X$.
It follows that these ideals  only depend on the embedding $i\colon X\to Y$.
We call them the {\it structure ideals}  on $X$ associated with the given embedding.

Given a Hermitian resolution $\big (\Ok(E_\bullet^g), g\big )$ of $\Ok^Y/\J_X$,
let $\sigma_k$ be the pointwise minimal inverse of $g^k$. If (after
resolution of singularities) $\Fitt_0 g^k$ is principal, generated by the
holomorphic 
section $s$, Lemma~2.1 in
\cite{AW1} asserts that $s\sigma_k$ is smooth. Thus
$i^*\sigma_{p+k}=:\sigma^k$ is well-defined and semi-meromorphic on
$X$.

\begin{lma}\label{purr1}
Assume that $\a_\ell$ and $\a'_\ell$ are the structure ideals  associated with the embeddings
$i: X\to Y$ and $i': X\to Y'$, respectively. Then for each $\ell\ge 1$, 
\begin{equation}\label{rodhake}
\a_{\ell}\cdots \a_{n-1}\subset \a'_\ell.
\end{equation}
\end{lma}

Since the zero set of $\a_{k+1}$ is contained in the zero set of $\a_k$ it follows
that the zero set of $\a_\ell$, which is $X^\ell$ as defined
in Section~\ref{fundform}, coincides with the zero 
set of $\a'_\ell$. It follows that $X^\ell$ is independent of the embedding.

\begin{proof}
Given $i\colon X\to Y$ and  a point $x\in X$ there is a \nbh $\V\subset X$ such that
the restriction to $\V$ of $i$ factorizes as
\begin{equation}\label{baddare}
\V \stackrel{j}{\to}\widehat \Omega\stackrel{\iota}{\to}\widehat \Omega\times\B_M=:\Omega,
\end{equation}
where $j$ is a minimal (and therefore
basically unique) embedding at $x$,  $\B_M\subset\C^M_w$ is a ball centered at $0$, $\iota$ 
is the trivial embedding $z\mapsto (z, 0)$ if $z$ are coordinates in $\widehat \Omega$, 
and $\Omega$ is a \nbh of $x$ in $Y$.  Let now $\big (\Ok(E^{\hat
  g}_\bullet), \hat g\big )$ be a minimal Hermitian  
resolution of $\Ok^{\widehat \Omega} / \J_\V$ at
$x$ in $\widehat \Omega$ and assume that $\hat p$ is the codimension of $\V$ in $\widehat \Omega$.
Thus $p=\hat p+M$, where as before $p$  is the codimension of $X$ in $Y$.

Let $(E^w,\delta_w)$ be the Koszul complex
generated by $w=(w_1,\ldots,w_M)$, cf., Example~\ref{koszulex}.
The sheaf complex associated with the
product complex  $E^{\hat g}\otimes E^w$ 
with mappings
$g=\hat g(z)+\delta_w$, cf.,  Section~\ref{tenprod},
provides a (minimal) resolution of $\Ok^\Omega/\J_X$ in $\Omega$,
see \cite[Remark~8]{A12}.
Notice that  $g^{p+\ell}$ is the mapping
\begin{multline*}
(E^{\hat g}_{\hat p+\ell}\otimes E^w_M)\oplus (E^{\hat g}_{\hat p+\ell+1}\otimes E^w_{M-1})\oplus
\cdots
\oplus
(E^{\hat g}_{\hat p+\ell+M}\otimes E^w_{0})\stackrel{\hat g(z)+\delta_w}{\longrightarrow}\\
(E^{\hat g}_{\hat p+\ell-1}\otimes E^w_M)\oplus
(E^{\hat g}_{\hat p+\ell}\otimes E^w_{M-1})\oplus\cdots\oplus 
(E^{\hat g}_{\hat p+\ell+M-1}\otimes E^w_{0}).
\end{multline*}
Since $w=0$ on $X$, the restriction of  $g^{p+\ell}$ to $X$
splits   into the direct  sum of  the  separate mappings
$$
\hat g^{\hat p+\ell+j}\colon E^{\hat g}_{\hat p+\ell+j}\otimes E^w_{M-j}
\to E^{\hat g}_{\hat p+\ell+j-1}\otimes E^w_{M-j},
 \quad j=0,1,\ldots, M.
$$
Since the optimal rank $r_{p+\ell}$ of $g^{p+\ell}$ is attained at every point on  $X_{\text{reg}}$,
it follows that
$r_{p+\ell}=\hat r_{\hat p+\ell}+\hat r_{\hat p+\ell+1}+\cdots+ \hat r_{\hat p + M}$, where
$\hat r_k$ is the generic rank of $\hat g^k$. 
Therefore,  the restriction to
$X$  of $\Fitt_0 g^{p+\ell}$ is equal to (the restriction to $X$ of)
the product ideal
\begin{equation*}
\Fitt_0\hat g^{\hat p+\ell}\cdot \Fitt_0\hat g^{\hat p+\ell+1}\cdots
\Fitt_0\hat g^{\hat p+\ell+M}.
\end{equation*}
Since $X$ has pure dimension,  $\Fitt_0 \hat g^k$ is trivial 
for $k\geq \hat p + n= \dim\widehat \Omega$, and thus if 
$\hat\a_\ell$ are the structure ideals associated with $j:\V\to \widehat \Omega$,
\begin{equation}\label{stork1}
\a_\ell=\hat \a_\ell\cdots\hat\a_{\min(n-1,\ell+M)}.
\end{equation}
Hence
\begin{equation}\label{stork}
\hat\a_\ell\cdots\hat\a_{n-1}\subset\a_\ell\subset\hat\a_\ell.
\end{equation}
By the same argument, since $i'$ factorizes as 
$
\V \stackrel{j}{\to}\widehat \Omega\stackrel{\iota'}{\to}\widehat \Omega\times\B_{M'}, 
$
at least if $\V$ is small enough,  
$
\a'_\ell=\hat \a_\ell\cdots\hat\a_{\min(n-1,\ell+M')},$
and so
\eqref{stork} holds at $x$ for $\a'_\ell$ instead of $\a_\ell$. 
Combining
we see that  \eqref{rodhake} holds in a \nbh of $x$. Since $x\in X$ is arbitrary,
the inclusion holds globally on $X$.
\end{proof}

\begin{lma} \label{purr21}
There is a smooth modification $\tau\colon\widetilde X\to X$ and a 
holomorphic section $\eta_0$ of a line bundle
$S_0\to\widetilde X$, which is nonvanishing in $\widetilde X\setminus
\tau^{-1}X_{\text{sing}}$, with the following properties:  
If $i\colon X\to Y$ is an embedding, $\dim  Y=N$, $p=N-n$,
and $\big (\Ok(E^g_\bullet),g\big )$ is a Hermitian locally free  resolution
of $\Ok^Y/\J_X$, then:
 
\smallskip
\noindent (i) all the ideals 
$\tau^*\a_\ell$, $\ell=1,\ldots,n-1$,  are principal,

\smallskip
\noindent (ii)
the subbundles $\Im \tau^*i^*g^{p+\ell}\subset \tau^*i^*E_{p+\ell-1}, \ \ell=1,\ldots,n-1,$ 
a~priori defined over $\widetilde X\setminus\tau^{-1}X^\ell$, all have holomorphic extensions
to $\widetilde  X$,

\smallskip
\noindent (iii) if $\omega=\omega_0+\cdots +\omega_n $ is the induced structure form, then $\eta_0\tau^*\omega_0$ is smooth.
\end{lma}

For the proof we will need the following, probably well-known,
result.

\begin{lma}\label{utvidgning} 
Let $E, Q$ be holomorphic vector
bundles over $X$ and let $g:E\to Q$ be a holomorphic morphism. Let
$Z\subset X$
be the analytic set where $g$ does not have optimal rank. 
There is a (smooth) modification $\pi:\widetilde
X\to X$ such that the subbundle $\pi^* \Im g\subset \pi^* Q$,
a priori defined in $\widetilde X\setminus \pi^{-1}Z$, has a
holomorphic extension to $\widetilde X$. 
\end{lma}
\begin{proof}
Let $G\colon Q\to F$ be a morphism such that $F$ is a direct sum of
line bundles, say $S_1, \ldots, S_r$,  
and
$$
\Ok(F^*)\stackrel{G^*}{\longrightarrow}\Ok(Q^*)\stackrel{g^*}{\longrightarrow}
\Ok(E^*)
$$
is exact, cf., \cite[Proposition~3.3]{AS}; we write $G=(G_1,\ldots,
G_r)$, where $G_j:Q\to S_j$. It then follows that 
\begin{equation}\label{hundvalp}
E\stackrel{g}{\longrightarrow}Q\stackrel{G}{\longrightarrow}F
\end{equation}
is pointwise exact in $X\setminus Z$. 
Therefore, 
$$
\Im g=\Ker G=\cap_j \Ker (Q\stackrel{G_j}{\to} S_j)
$$ 
on 
$X\setminus Z$. 

To prove that $\Ker G$ has a holomorphic extension, let us first
assume that $F$ has rank $1$, so that $G$ defines an ideal sheaf
$\J_G\subset \Ok^X$. Also, let us assume that $X$ is connected; if
not we just 
consider each connected component separately.
If $G$ is identically zero we define $K:=Q$. 
Otherwise let $\pi: \widetilde X\to X$ be the blow-up of $X$ along $\J_G$,
let $D$ be the corresponding divisor, and let $\Ok(-D)$ be the line
bundle defined by $D$. Then (the pullback to $\widetilde X$ of) $G$ is of the
form $G^0 G'$, where $G'$ is a nonvanishing mapping
$Q\to F\otimes \Ok(-D)$ and $G^0:F\otimes \Ok(-D)\to F$ is
generically invertible. Thus $K:=\Ker G'$ is a holomorphic
subbundle of $Q$, and it generically coincides with $\pi^*\Ker G$.

For the general case, we proceed by induction: We let $K_1$ be an
extension of $\Ker G_1$ as above. Then we let $K_2\subset K_1$ be an
extension of the kernel of $G_2|_{K_1}\colon K_1\to  S_2$. 
Proceeding in this way we find subbundles 
$K_r\subset \cdots \subset K_1\subset Q$, such that $K_j$ generically
coincides with $\Ker G_1\cap \cdots \cap \Ker G_j$ on $X$. In particular  $K_r$ coincides with 
$\Im g $ generically on $X$, and so we have found a holomorphic
extension of $\Im g $.
\end{proof}

\begin{proof}[Proof of Lemma~\ref{purr21}]

Let us first fix an embedding $i\colon X\to Y$  and a Hermitian locally free resolution 
$\big (\Ok(E^g_\bullet),g\big )$, and show that there are $\tau\colon\widetilde X\to X$ and
$\eta_0$ such that $(i)-(iii)$ hold. 
To begin with, by resolution of singularities, we can find a smooth modification $\hat\tau\colon \widehat X\to X$ such that
all $\hat\tau^*\a_\ell$ are principal, so that $(i)$ holds.  
Next, by repeated use of Lemma ~\ref{utvidgning} we can find a modification $\tau:
\widetilde X\to \widehat X$ so that the subbundles $\Im \tau^*\hat\tau^*i^* g^{p+\ell}$
have holomorphic extensions.

Let us now consider $(iii)$. According to Proposition~3.3  in
\cite{AS}, $\omega_0$ is of the form 
$$
\omega_0=\sigma_G h,
$$
where $h$ is holomorphic in the Barlet-Henkin-Passare  sense, i.e.,
$\dbar h=0$ on $X$, 
$G$ is a holomorphic map from $E_p$ to a vector bundle $F$, and 
$\sigma_G \colon F\to E_p$  is the inverse of $G$ in $X\setminus X^1$ with pointwise minimal norm,
vanishing on the orthogonal complement of $\Im G$. 
After further resolving we may  assume that $\tau$ 
is chosen so that
also (the pullback of) the ideal $\a_G$ is principal in  $\widetilde
X$, say, generated by the section $s_G$. Then, by \cite[Lemma~2.1]{AW1},
$s_G \sigma_G$ is smooth,  cf., the text preceding
Lemma \ref{purr1}. 
Since $h$ is meromorphic, there is a section $\eta_0$ of a line bundle $S_0\to \widetilde X$ 
such that $\eta_0\tau^*\omega_0$ is smooth. 
We may also assume that $\tau^{-1}X_{\text{sing}}$ is a
divisor, so that $\tau^* h$ is meromorphic with poles contained in
$\tau^{-1}X_{\text{sing}}$. Since the variety of $\a_G$ is contained
in $X_{\text{sing}}$ it follows that we can choose $\eta_0$ to be
nonvanishing in $\widetilde X\setminus \tau^{-1}X_{\text{sing}}$.

\smallskip


We will prove that with the choice of $\tau:\widetilde X \to X$ and $\eta_0$
above, $(i)-(iii)$, in
fact, hold for any choice of embedding and Hermitian resolution. 
%
%
We first keep the
embedding  $i\colon X\to Y$ 
and vary the Hermitian resolution. 
Pick a Hermitian locally free
resolution 
$\big (\Ok(E^{\tilde g}_\bullet), \tilde g\big )$ of $\Ok^Y/\J_X$, fix
a point $x\in X$, and choose a
minimal Hermitian locally free
resolution resolution $\big (\Ok(E^{g'}_\bullet), g'\big )$ at $x$. 


\noindent 
\textbf{Claim 1}:  
\emph{$(i)-(iii)$ hold for  $\big (\Ok(E^{\tilde g}_\bullet), \tilde
  g\big )$ at $x$ if and only if they hold for $\big
(\Ok(E^{g'}_\bullet), g'\big )$ at $x$.}  


\noindent 
Since the choices of $\big (\Ok(E^{\tilde g}_\bullet), \tilde g\big )$
and $x$ are 
arbitrary, it follows that $(i)-(iii)$ hold for any 
Hermitian resolution since they hold for $\big (\Ok(E^{ g}_\bullet),
g\big )$. 

\begin{proof}[Proof of Claim 1]
It is well-known that $(E^{g'}_\bullet, g')$ is a direct summand in
$(E^{\tilde g}_\bullet, \tilde g)$, i.e., 
 there is 
a decomposition  $(E^{\tilde g}_\bullet= E^{g'}_\bullet\oplus E^{g''}_\bullet$,
$\tilde g=g'\oplus g'')$, 
where the complex  $(E^{g''}_\bullet, g'')$ is pointwise exact, see, e.g., \cite[Theorem~20.2]{Eis}. 
Since $\Im \tilde g^{p+\ell}=\Im (g')^{p+\ell}\oplus\Im(g'')^{p+\ell}$ it follows that
(the pullback to $\widetilde X$ of) $\Im \tilde g^{p+\ell}$
has a holomorphic extension if and only if $\Im (g')^{p+\ell}$ has one,
for $\ell\ge 1$, i.e., $(ii)$ holds at $x$ for one of the Hermitian
resolutions $\big (\Ok(E^{\tilde g}_\bullet), \tilde g\big )$ and $\big (\Ok(E^{g'}_\bullet), g'\big )$ if and only if it
holds for the other one. 
From  this decomposition it also follows immediately  that $\Fitt_0 \tilde g^{p+\ell}=\Fitt_0 (g')^{p+\ell}$,
so that the structure ideals $\a_\ell$ are 
independent of the Hermitian resolution, cf.\ the beginning of this section. In particular, $(i)$ holds for
one of the resolutions if and only if it
holds for the other one. 

Let  $\tilde \omega$ and $\omega'$ be the structure forms associated with $\big
(\Ok(E^{\tilde g}_\bullet), \tilde g\big )$ and $\big
(\Ok(E^{g'}_\bullet), g'\big )$, respectively. 
Then  $\omega'$ 
can be considered as a structure form associated with
$\big (\Ok(E^{\tilde g}_\bullet), \tilde g\big)$ but with a
Hermitian  metric that respects the direct sum, cf., \cite[Section~4]{AW1} and 
\eqref{poseidon}. 
Moreover $\tilde \omega_0=\pi
\omega_0'$, where $\pi$ is the orthogonal projection onto
the orthogonal complement (with respect to the first metric) of $\Im
\tilde g^{p+1}$ in $E_p|_X$ over $X\setminus X^1$, and $\omega'_0=\pi'
\tilde \omega_0$, where $\pi'$ is the orthogonal projection onto
the orthogonal complement (with respect to the second metric) of $\Im
(g')^{p+1}$ in $E_p|_X$ over $X\setminus X^1$, cf., the proof of
Theorem~4.4 in \cite{AW1}. 
If $(ii)$ holds for (at least one of) the resolutions, then
$\tau^*\pi$ and $\tau^*\pi'$ are smooth and it follows that
$\eta_0\tau^*\omega_0$ is smooth if and only if
$\eta_0\tau^*\omega_0'$ is, i.e., $(iii)$ holds for one of the
resolutions if
and only if it holds for the other one. 



\end{proof}

Next, we will vary the embedding of $X$. 
Pick an embedding $i^\sharp: X\to Y^\sharp$ and $x\in X$. 
Then, in a neighborhood $\mathcal V$ of $x$, $i^\sharp$
factorizes as \eqref{baddare}, where now $\Omega$ is a neighborhood of
$x$ in $Y^\sharp$.

\noindent 
\textbf{Claim 2}:\emph{ 
$(i)-(iii)$ hold for Hermitian resolutions 
of $\Ok^{\Omega}/\J_X$ at $x$ if and only if they hold for Hermitian
resolutions of $\Ok^{\widehat \Omega}/\J_X$ at $x$.}

\noindent 
Since $i^\sharp$ and $x$ are 
arbitrary and all embeddings of $X$ factor through the minimal
embedding $j$ in a small 
neighborhood of $x$ it then follows that $(i)-(iii)$ hold for
any embedding of $X$.


\begin{proof}[Proof of Claim 2]
 Let $\big (\Ok(E_\bullet^{\hat g}), \hat g\big )$ be a Hermitian minimal 
resolution of $\Ok^{\widehat \Omega}/\J_X$. 
Then, using the notation from
the proof of Lemma \ref{purr1}, $\big (\Ok(E^{\check g}):=\Ok(E^{\hat g})\otimes E^w,
\hat g + \delta_w=:\check g\big )$ is a (minimal) resolution of
$\Ok^\Omega/\J_X$. 
From Claim ~1 we know that it suffices to show that $(i)-(iii)$ hold for
$\big (\Ok(E_\bullet^{\hat g}), \hat g\big )$ if and only if they hold
for  $\big (\Ok(E_\bullet^{\check g}), \check g\big )$.

The residue current associated to $\big (\Ok(E_\bullet^{\check g}),
\check g\big )$ is equal to $R^{\hat g(z)}\wedge R^w$, see
\cite[Remark~4.6]{A12}. 
Since a product of local ideals is principal if and only each of its factors
is principal it follows from
\eqref{stork1} that $\tau^*\check\a_\ell$ are principal for $\ell=1,\ldots,n-1$ if and only
if $\tau^*\hat\a_\ell$ are principal for $\ell=1,\ldots,n-1$, where
$\check \a_\ell$ denotes the structure ideal associated with $i^\sharp$,
i.e.,
$(i)$ holds for $\big (\Ok(E_\bullet^{\hat g}), \hat g\big )$ if and
only if it holds for $\big (\Ok(E_\bullet^{\check g}), \check g\big )$. 
Moreover, since the restriction
of $\check g^{p+\ell}$ to $X$ is a direct sum of restrictions of $\hat
g^{\hat p+\ell +j}$, cf.,
the proof of Lemma~\ref{purr1}, it follows that (the pull-back to
$\widetilde X$ of) $\Im \check g^{p+\ell}$, $\ell\ge 1$, have holomorphic extensions
if and only if $\Im \hat g^{\hat p+\ell}$, $\ell\ge 1$, have, so that
$(ii)$ holds for one of the resolutions $\big (\Ok(E_\bullet^{\hat g}), \hat g\big )$ and $\big (\Ok(E_\bullet^{\check g}), \check g\big )$ if and only if it holds for the
other one. 

Since $w$ are just the coordinate functions in $\C^M$, the Poincar\'e-Lelong formula asserts 
that 
\[
R^w_M\w dw_1\w\ldots\w dw_M=(2\pi i)^M[w=0], 
\]
where $[w=0]$ is
the current of integration over the affine set $\{w=0\}$. 
Let $\widehat N=\dim \widehat\Omega$, and let $\widehat \omega$ denote the structure form in 
$\mathcal V$ associated with $R^{\hat g(z)}$, so that
$j_*\widehat\omega=R^{\hat g}\w dz_1\w\ldots\w dz_{\widehat N}$. Then, 
\begin{multline*}
i_*\widehat\omega =\iota_*R^{\hat g}\w dz_1\w\ldots\w dz_{\widehat N} = 
R^{\hat g}\w dz_1\w\ldots\w dz_{\widehat N}\w[w=0]\sim \\
R^{\hat g}\w R^w\w dw_1\w\ldots\w dw_{M}\w dz_1\w\ldots\w dz_{\widehat N},
\end{multline*}
where $\sim$ denotes ``equal to a nonzero constant times''. 
We conclude, cf., \eqref{poseidon}, that $\widehat\omega$ is also a structure form
associated with a Hermitian resolution of $\Ok^\Omega/\J_X$. 
From the proof of Claim ~1 we know that that if we have two Hermitian
resolutions of $\Ok^\Omega/\J_X$, and that $(ii)$ holds (for at least
one of the  
resolutions), then $(iii)$ holds for one of the resolutions if and only it holds
for the other resolution. Thus, provided that $(ii)$ holds,
$\eta_0\tau^*\check\omega_0$ is smooth if and only if
$\eta_0\tau^*\widehat\omega_0$ is, where $\check\omega_0$ denotes the
structure form associated with $\big (\Ok(E_\bullet^{\check g}),
\check g)$, i.e., $(iii)$ holds for $\big (\Ok(E_\bullet^{\hat g}),
\hat g\big )$ if and only if it holds for $\big (\Ok(E_\bullet^{\check g}),
\check g\big )$.

\end{proof} 

This concludes the proof of Lemma ~\ref{purr21}: With the choice of $\tau: \widetilde X\to X$
and $\eta_0$ made above, $(i)-(iii)$ hold for all embeddings $i:X\to Y$
and all Hermitian resolutions of $\Ok^Y/\J_X$.

\end{proof}

We can now conclude the proof of  Proposition~\ref{invarians}. 
Let $\tau\colon\widetilde X\to X$ and  $\eta_0$ be as in 
Lemma~\ref{purr21}.
Fix an embedding
$i':X\to Y'$ and let $s'_1,\ldots, s'_{n-1}$ be sections on $\widetilde  X$  defining
(the pull-back to $\widetilde X$ of)
the ideals $\a_1, \ldots, \a_{n-1}$. Let $\eta_\ell=s'_\ell\cdots s'_{n-1}$,
$\ell\ge 1$, and
$\eta=\eta_0\eta_1\cdots\eta_{n-1}$. Note that $s'_\ell$ is
nonvanishing outside $\tau^{-1}X_{\text{sing}}$ so that $\eta$ is
nonvanishing in $\widetilde X\setminus \tau^{-1}X_{\text{sing}}$ if
$\eta_0$ is. 
We claim that $\eta\tau^*\omega$ is smooth for any structure form
$\omega$ on $X$. To see this, 
let $\omega$ be the structure form associated with
an embedding $i:X\to Y$ and a Hermitian locally free resolution $\big (\Ok(E_\bullet^g),
g\big )$. Assume that (the pullbacks of) the corresponding structure ideals are defined by
sections $s_1,\ldots, s_{n-1}$. Outside $X^\ell$, $\omega_\ell=\alpha^{\ell}\omega_{\ell-1}$, where
$\alpha^\ell=\1_{X_{\text{reg}}}\dbar\sigma^\ell$, cf.,
\eqref{snorlik} and (the notation in) 
\cite[Section~2]{AW1}. By \cite[Lemma~2.1]{AW1}, 
$s_\ell\tau^*\sigma^\ell$ is smooth in $\widetilde X$. 
Thus, since $\omega_\ell$ has the SEP,  $\eta_0 s_1\cdots
s_\ell\omega_\ell$ is smooth, and so $\eta_0 s_1\cdots
s_{n-1}\omega$ is smooth. 
By Lemma ~\ref{purr1}, $s_\ell$
divides $\eta_\ell$ and hence the claim follows. 
This concludes the proof of Proposition~\ref{invarians}.

\begin{remark}\label{torsdag}


Let $\omega'$ be a structure form on $X$ associated with a given embedding $i':X\to Y'$.
From the proof above, using the notation in the proof, it
follows that the section $\eta':=\eta_0 s_1'\cdots s_n'$ satisfies
that $\eta'\tau^*\omega'$ is smooth. 
If $i':X\to Y'$ is the fixed embedding in the last part of the proof, then  
\[
\eta=\eta_0 (s'_1)\cdots (s_\ell')^\ell\cdots (s_{n-1}')^{n-1}=
(s'_2)\cdots (s_\ell')^{\ell-1}\cdots (s_{n-1}')^{n-2}\eta'. 
\]
In particular, $\eta$ divides $(\eta')^{n-1}$. 


\end{remark}

\section{Global division problems and residues}\label{divandres}

In this section we will discuss a method for solving division problems
on $X$ using residue theory, which originates from 
\cite{A3}. 
Throughout the section, 
\eqref{ecomplex} is a generically exact Hermitian complex over $X$  
and $\phi$ is a global holomorphic section of $E_0$. 





\smallskip

Let us first assume that $X$ is smooth and that $R^f\phi=0$. 
As we have seen in Section ~\ref{svart}, then  $\nabla_f (U^f\phi)=\phi$.
If the double complex $\M_{\ell,k}=\Cu_{0,k}(X,E_\ell)$ is exact in the $k$-direction except at
$k=0$,  then it follows, cf., \eqref{bus1},  that there is
a global holomorphic solution to $f^1 q=\phi$. Let us  see more precisely
what is needed. 
Notice that $U^f_{\min(M,n+1)}\phi$ is automatically $\dbar$-closed.
Since  $X$ is smooth, by the Dolbeault isomorphism for currents
it is possible to successively solve the equations
$$
\dbar w_{\min(M,n+1)}=U^f_{\min(M,n+1)}\phi, \quad \dbar w_k= U^f_k\phi -f^{k+1} w_{k+1},\ 1\le k<\min(M,n+1),
$$
if
\begin{equation}\label{cohovillkor}
H^{k-1}(X,\Ok(E_k))=0, \quad  1\le k\le\min(M,n+1).
\end{equation}
Then
$$
q:=U^f_1\phi-f^2w_2
$$
is a holomorphic solution to $f^1q=\phi$. To sum up we have

\begin{prop}\label{glatta}
Assume that $X$ is smooth and $\phi$ is a holomorphic section of
$E_0$. If $R^f\phi=0$ and \eqref{cohovillkor} holds, then
there is a global holomorphic section $q$ of $E_1$ such that $f^1 q=\phi$.
\end{prop}

\begin{remark}\label{giraff}
The essence in Proposition~\ref{glatta} is that the vanishing of $R^f\phi$ not only implies that $\phi$
belongs to the sheaf $\J_f\otimes E_0$ but is in the image of
$\Gamma(X,E_1)\to\Gamma(X,\J_f\otimes E_0)$, provided that \eqref{cohovillkor} is
fulfilled. 
In general this map is not surjective even if \eqref{cohovillkor} is fulfilled. 
\end{remark}





We will now look for analogous results when $X$ is nonsmooth. 
Since
we have no access to a  $\dbar$-theory for currents on $X$, we need to
embed $X$ in a smooth (projective) manifold. 
We start by considering a special case that is needed for the proof of
Theorem~A, namely the case when $X$ is embedded in
$\P^N$, $(E^f_\bullet,f)$ is the Koszul complex generated by homogeneous
forms $f_j$ of degree $d$, i.e., global sections of $\Ok(d)\to \P^N$, and
$\phi$ is a section of $\Ok(\rho)\to \P^N$. 
 Let $(E^g_\bullet,g)$ be an exact 
Hermitian complex on $\P^N$ associated to $X$ as in
Section~\ref{lotta} of length $\leq N$. 
If $R^f\w R^g \phi=0$, then $v=(U^g+U^f\w R^g)\phi$ is a global current solution
to $\nabla v=\phi$ in $\P^N$, see Section ~\ref{boork}, and, provided
that we can solve a sequence of $\dbar$-equations on $\P^N$, we get a
global solution to $f\cdot q+g\cdot q'=\phi$ on $\P^N$, and thus a 
solution $q$ to
$f\cdot q=\phi$ on $X$. 
However, see,  e.g., \cite{Dem}, 
\begin{equation}\label{pnko}
H^k(\P^N,\Ok(\ell))=0 \quad\quad  \text{if}\quad  \ell\ge -N \quad \text{or}\quad k<N
\end{equation}
so the only possible obstruction is the equation
\begin{equation}\label{enda}
\dbar W=U_{N+1}\phi,
\end{equation}
where $U=U^f\w R^g+U^g$. 
Since $(E^g_\bullet,g)$ ends at level $N$, $U^g_{N+1}=0$. 
Moreover, $R^g_k=0$ for $k<N-n$ by the dimension principle, so 
\begin{equation}\label{sumolin}
U_{N+1}=\sum_{k=1}^{\min(m,n+1)}U^f_k\w R^g_{N+1-k},
\end{equation}
cf., Section~\ref{tenprod}.
The term corresponding to $k$ takes values in a direct sum of line bundles
$\Ok(-dk-d_{N+1-k}^i)$.  In view of \eqref{pnko}, one can  solve \eqref{enda}
if  $\rho\ge dk+d_{N+1-k}^i-N$ for all $i$ and $k=1,2,\ldots,\min(m,n+1)$.
Notice that, cf., \eqref{spasm}, 
$$
 dk+d_{N+1-k}^i-N =dk+\big(d_{N+1-k}^i-(N+1-k)\big)+1-k\le (d-1)k + \reg X .
$$
It follows that \eqref{enda} is solvable as soon as 
\begin{equation}\label{enda2}
\rho\ge (d-1)\min(m,n+1)+\reg X.
\end{equation}
Summing up we have: 


\begin{lma}\label{direkt} 
 If $\rho$ satisfies \eqref{enda2}
and $\phi$ is a section of $\Ok(\rho)$ on $\P^N$ such that
$R^f\w R^g\phi=0$, then there are global sections $q_j$ of
$\Ok(\rho-d)$ such that $f_1q_1+\cdots +f_mq_m=\phi$ on ~$X$.
\end{lma}

\begin{remark}\label{indirekt} 
To be more precise, only terms where $N+1-k\le M$ occur in \eqref{sumolin},
where $M$ is the length of $(E^g_\bullet,g)$.
If for instance $X$ is Cohen-Macaulay,
i.e., the ring $\S/J_X$ is Cohen-Macaulay, and $(E^g_\bullet,g)$ is of
minimal length, then $M=N-n$ so that 
$k\ge n+1$. 
If in addition $m\le n$ thus  $U_{N+1}$ vanishes, so there is no 
cohomological obstruction at all.
\end{remark} 

\smallskip 

In general it is not possible to find an embedding of $X$ into a
smooth manifold $Y$ such that $(E_\bullet^f, f)$ and $\varphi$ extend
holomorphically to $Y$. 
For our next
result  (Theorem~\ref{fund}), we will still assume that
$(E^f_\bullet,f)$ extends. 
%
%
%
%
%
As a substitute for a holomorphic extension of $\phi$ we will use
a \emph{$\nabla_g$-closed extension} $\Phi$ of $\phi$ to $Y$. 
If $i: X\to Y$ is an embedding of $X$ into a projective manifold $Y$,
$\big (\Ok(E_\bullet^g), g\big )$ is a Hermitian resolution of
$\Ok^Y/\J_X$, and $\phi$ is a global holomorphic section on $X$ of a line bundle
$S\to Y$, then we say that 
a global smooth
section $\Phi=\sum_{\ell\ge 0}\Phi_\ell$ of
$\oplus_\ell\E_{0,\ell}(E^g_\ell\otimes S)$ on $Y$  is a
$\nabla_g$-closed extension of $\phi$ if 
$\nabla_g \Phi=0$ on  $Y$ and  $i^*\Phi_0=\phi$.  Recall that $E^g_0\simeq\C$ is a trivial line bundle.

\begin{lma}\label{ut}

\noindent (i) 
Any $\phi$ admits a $\nabla_g$-closed extension. 

\smallskip
\noindent (ii)  
$\Phi$ is a  $\nabla_g$-closed extension of $\phi$ 
if and only if
\begin{equation}\label{pontus}
\Phi- R^g\phi=\nabla_g w
\end{equation}
for some current $w$.
\end{lma}

One  can obtain a $\nabla_g$-closed  extension $\Phi$ of $\phi$  quite elementarily
by piecing together local holomorphic extensions, due to
the exactness of $\big (\Ok(E^g_\bullet),g\big )$.  However, we prefer an argument that
also relates to residue calculus as in $(ii)$, and we also think
that Lemma~\ref{ut}~$(ii)$  may be of independent interest.

\begin{proof}[Proof of Lemma ~\ref{ut}] 
As noted in Section~\ref{boork}, $R^g\phi$ is a well-defined $\nabla_g$-closed current in $Y$. In view of
Proposition~\ref{urlaka} there is a smooth $\nabla_g$-closed $\Phi$ such that
\eqref{pontus} holds for some current $w$.  Thus $(i)$ follows from $(ii)$.

Assume that $\Phi$ is a smooth  extension of $\phi$ as in (i).
From \eqref{urformel} we have that $\nabla_g (U^g\w\Phi)=\Phi-R^g\w\Phi$.
Since $\big (\Ok(E^g_\bullet),g\big )$ is exact, $(R^g)^\ell=0$ for $\ell\ge 1$,
cf., Section~\ref{bef},  and hence $R^g\w\Phi=R^g\Phi_0=R^g\phi$, since
$\Phi_0=\phi$ on  $X$, i.e., $i^*\Phi_0=\phi$ on $X$. 
Thus
$$
\nabla_g (U^g\w\Phi)=\Phi-R^g\phi.
$$

Conversely, assume that $\Phi$ is smooth and \eqref{pontus} holds. 
Then clearly $\nabla_g\Phi=0$.  We have to prove that
$\Phi_0=\phi$ on $X$. Notice that this is a local statement. Given a point  on $X$ there is a \nbh
$\U$ where we have holomorphic extension $\hat \phi$ of $\phi$.  Then
$\nabla_g(U^g\hat\phi)=\hat\phi-R^g\hat\phi=\hat\phi-R^g\phi$ in $\U$.  Thus
$
\nabla_g(w-U^g\hat\phi)=\Phi-\hat\phi.
$
By Proposition~\ref{urlaka} there is a smooth $\xi$ such that
$
\nabla_g\xi =\Phi-\hat\phi.
$
It follows that
$
g^1\xi_1=\Phi_0-\hat\phi
$
and hence $\Phi_0=\hat\phi=\phi$ in $\U\cap X$.
\end{proof}

We have the following analogue of Proposition~\ref{glatta}. 

\begin{thm}\label{fund}
Let $i\colon X\to Y$ be an embedding of $X$ in a projective  manifold $Y$,   let
$\big (\Ok(E^g_\bullet), g\big )$ be a locally free Hermitian resolution of 
$\Ok^Y/\J_X$ in $Y$, and let $\omega$ be an associated structure
form on $X$.

Let \eqref{ecomplex} be a Hermitian complex over (an open \nbh $\U$ of $X$ in) $Y$,
and let  $R^f\w\omega$ be the associated
residue current.  Moreover let $\phi$ be a global section of $E_0$ on $X$.
\smallskip

\noindent (i) If  $R^f\w\omega\phi=0$,
then there is a global smooth solution $W$ on $X$ to
\begin{equation}\label{polka}
\nabla_f W =\phi.
\end{equation}

\smallskip
\noindent (ii) If \eqref{polka} has a global smooth solution on $X$ and \eqref{cohovillkor} holds,
then there is a global holomorphic section $q$ of $\Ok(E_1)$ such that $f^1 q=\phi$ on $X$.
\end{thm}

With minor  modifications of the proof below we get the following more general version
of Theorem~\ref{fund}:

\smallskip

\noindent 
{\it With the general hypotheses of Theorem~\ref{fund}, assume that
 $\phi$ is a global holomorphic section of $E_\ell$ such that $f^\ell \phi=0$.

\smallskip
\noindent (i) If $R^{\ell}\w\omega\phi=0$ then  there is a smooth global solution to
\eqref{polka}.

\smallskip
\noindent (ii) If \eqref{polka} has a smooth solution and
$$
H^{0,k-1-\ell}(X,\Ok(E_k))=0, \quad\ell+1\le k\le\min(M,n+1+\ell),
$$
then there is a global holomorphic section $q$
of $E_{\ell+1}$ such that $f^{\ell+1} q=\phi$.}

\begin{remark} If we just have a current solution to $\nabla_f T=\phi$ on $X$ it does not
follow that there is a holomorphic solution, not even locally. In fact, if $X$ is non-normal,
there are holomorphic $f$ and $\phi$ such that $\dbar(\phi/f)=0$ but $U=\phi/f$ is  not
holomorphic. Thus  $(f-\dbar) U=\phi$ but $\phi$ is not in the ideal $(f)$.
If $X$ is normal but nonsmooth, there are similar examples with more generators, see 
\cite{Lark}. 
\end{remark}

\begin{proof}[Proof of Theorem~\ref{fund}]
Recall from Section~\ref{boork} that  $R^f\w\omega\phi=0$ implies that $R^f\w R^g\phi=0$.
Let $\Phi$ be a $\nabla_g$-closed smooth  extension of $\phi$, as in
Lemma~\ref{ut}~(i),  to $Y$.
As in the proof of Lemma~\ref{ut},
$R^g\w\Phi=R^g\Phi_0=R^g\phi$. It follows that 
$R^f\w R^g\wedge\Phi=R^f\w R^g\phi=0$.
Hence, from \eqref{badanka}     we get, cf.,  Section~\ref{boork},
$$
\nabla_F [(U^f\w R^g+U^g)\wedge \Phi] =\Phi.
$$
By Proposition~\ref{urlaka}  we have a smooth solution $\Psi$ to
$\nabla_F \Psi=\Phi$ in $Y$; i.e.,
$$
F^1 \Psi_1=\Phi_0,\quad F^{k+1}\Psi_{k+1}-\dbar \Psi_k=\Phi_k, \ k\ge 1.
$$
If we let lower indices $(i,j)$ denote values in $E^f_i\otimes E^g_j$,
and notice that  $\Phi_k=\Phi_{0,k}$,
we see that
\begin{equation}\label{skrattmas}
f^1 \Psi_{1,0}+g^1 \Psi_{0,1}=\Phi_0,
\quad  f^{k+1} \Psi_{k+1,0}+g^1 \Psi_{k,1}-\dbar \Psi_{k,0}=0,\ k\ge 1.
\end{equation}
Since $\Psi$ is smooth we can define the forms $W_k=i^*\Psi_{k,0}$ on
$X$,    and \eqref{skrattmas} then implies that
$$
f^1 W_1=\phi, \quad  f^{k+1} W_{k+1}-\dbar W_k=0, \ k\ge 1.
$$
Thus (i) follows.

\smallskip 

The proof of (ii) is similar to the case when $X$ is smooth, cf.\ the
beginning of Section~\ref{divandres}:  
Assume that $W$ is a global smooth solution to \eqref{polka}. Then
$W_{\min(M,n+1)}$ is automatically $\dbar$-closed, and thus if
\eqref{cohovillkor} is satisfied we can successively solve the equations 
$$
\dbar \eta_{\min(M,n+1)}=W_{\min(M,n+1)}, \quad \dbar \eta_k= W_k -f^{k+1} \eta_{k+1},\ 1\le k<\min(M,n+1),
$$
where $\eta_k$ is in $\A_k$, see Section ~\ref{fineres}. 
Then $q:=W_1-f^2\eta_2$ 
is a holomorphic solution to $f^1q=\phi$.

\end{proof}

Note that the proof of (ii) above only depends on $X$ and not on the
embedding $i:X\to Y$.

It should be possible to express the $\nabla_F$-exactness of $\Phi$ in $Y$
by means of \v Cech cohomology, then make the restriction to $X$,
and rely on the vanishing of the relevant \v Cech cohomology groups
on $X$. In this way one could avoid the
reference to the sheaves  $\A_k$ over $X$.


\section{Integral closure, distinguished varieties and  residues}\label{dist}
Let  $f_1,\ldots, f_m$ be global holomorphic sections of the
ample Hermitian line bundle $L\to X$, and let  $\J_f$ be the coherent ideal sheaf they generate.  Let
$$
\blow\colon X_+\to X
$$
be the normalization of the blow-up
of $X$ along $\J_f$, and let $W=\sum r_j W_j$ be the exceptional divisor; here $W_j$ are
irreducible Cartier divisors. The images $Z_j:=\blow(W_j)$ are called
the {\it (Fulton-MacPherson) distinguished varieties}   associated with $\J_f$.
If $f=(f_1,\ldots,f_m)$ is considered as a section of $E^*:=\oplus_1^m L$, then
 $\blow^* f=f^0f'$, where $f^0$ is a section of the line bundle
$\Ok(-W)$  defined by $W$,  and $f'=(f_1',\ldots,f_m')$ is a nonvanishing section of
$\blow^* E^*\otimes \Ok(W)$, where  $\Ok(W)=\Ok(-W)^{-1}$.
Furthermore,
$
\omega_f:=dd^c\log|f'|^2
$
is a smooth first Chern form for $\blow^* L\otimes\Ok(W)$.

Recall  that
(a germ of) a holomorphic  function  $\phi$ belongs to the {\it integral closure}
$\overline{\J_{f,x}}$ of $\J_{f,x}$ at $x$
if $\blow^*\phi$ vanishes to order (at least) $r_j$ on  $W_j$
for all  $j$ such that $x\in Z_j$.  This holds
if and only if $|\blow^*\phi|\le C|f^0|$ (in a \nbh of the
relevant $W_j$), which in turn holds if and only if
$|\phi|\le C|f|$ in some \nbh of $x$.
Let $\overline{\J_f}$ denote the integral closure sheaf.
It follows that
\begin{equation}\label{pulka}
|\phi|\le C|f|^\ell \quad \text{if and only if}\quad \phi\in\overline{\J_f^\ell}.
\end{equation}
If  $X$ is smooth it follows that $\phi$ is in the integral closure,  if for each $j$,
$\phi$  vanishes to order $r_j$ at a generic point on $Z_j$.
See \cite[Section~10.5]{Laz}
for more details (e.g., the proof of Lemma~10.5.2). 
We will use the geometric estimate
\begin{equation}\label{elest}
\sum r_j \deg_L Z_j\le \deg_L X
\end{equation}
from  \cite{EL},
see also \cite[(5.20)]{Laz}.

\begin{lma}\label{omljud}
There is a number $\mu_0$,  only depending on $X$, such that if
\begin{equation}\label{bus}
|\phi|\le C|f|^{\mu+\mu_0},
\end{equation}
then $R^f\w\omega \phi=0$ if $\omega$ is a structure form of $X$ and $R^f$ is the residue current obtained from the Koszul complex of $f$.
If $X$ is smooth one can take $\mu_0=0$.
\end{lma}

This proposition (and its proof) is analogous to Proposition~4.1 in~\cite{ASS}; the important novelty here is that $\mu_0$ can be chosen uniform in $\omega$, which is ensured by Proposition~\ref{invarians}. However, for the readers convenience and future reference we discuss the proof. 

\begin{proof}
Let us first assume that $X$ is smooth and $\mu_0=0$, and that $\phi$ satisfies \eqref{bus}. Then $\omega$ is smooth so we 
have to show that $R^f\phi=0$. 
If $f\equiv 0$ on (a component of) $X$, then $R^f\equiv 1$ and
$\phi\equiv 0$, and thus $R^f\phi =0$. Let us now assume that $\codim Z^f\geq 1$. Then $R^f_0=0$ by the dimension principle. 
Let $\blow\colon  X_+\to X$ be the normalization of the blow-up
along $\J_f$ as above,  so that $\blow^*f=f^0f'$.
Using the notation in Example~\ref{koszulex}, then
$
\blow^*\sigma=(1/f^0)\sigma',
$
where $1/f^0$ is a meromorphic section of $\Ok(W)$ and $\sigma'$ is a smooth
section of $\blow^*E\otimes\Ok(-W)$. It follows that
$$
\blow^*(\sigma\w(\dbar\sigma)^{k-1})=\frac{1}{(f^0)^k}\sigma'\w(\dbar\sigma')^{k-1},
$$
and hence
$$
\blow^*R^{f,\lambda}_k=\dbar |f^0f'|^{2\lambda}\w
\frac{1}{(f^0)^k}\sigma'\w(\dbar\sigma')^{k-1},
$$
when $k\ge 1$. 
Since $f'$ is nonvanishing,  the value at $\lambda=0$ is precisely, see, e.g., \cite[Lemma~2.1]{A2},
\begin{equation}\label{stuga1}
R_k^+:=\dbar\frac{1}{(f^0)^k}\w \sigma'\w(\dbar\sigma')^{k-1}.
\end{equation}
Notice that 
\begin{equation}\label{stuga}
\blow_* R^+_k=R^f_k.
\end{equation}

Since $\phi$ satisfies \eqref{bus} for $\mu_0=0$, $|\blow^*\phi|\le C|f^0|^\mu$ and, since 
$X_+$ is normal it follows that
$\blow^*\phi$ contains a factor $(f^0)^\mu$. Therefore,
\begin{equation}\label{bula}
\blow^*\phi \dbar\frac{1}{(f^0)^k} =0, \quad k\le\mu,
\end{equation}
because  of \eqref{kalkyl}. 
Moreover, since $\sigma'\w(\dbar\sigma')^{k-1}$ is smooth on $X_+$,
it follows from \eqref{bula} and \eqref{stuga1} 
that 
$R^+_k\blow^*\phi=0$. Therefore, cf., \eqref{stuga},  $R^f_k\phi=\blow_*(R^+_k \blow^*\phi)=0$.

Notice that we could have used any normal modification $\pi\colon\widetilde X\to X$ such that
$\pi^*f$ is of the form $f^0f'$ in the proof so far.

\smallskip
Now consider a general $X$. Let us take a smooth modification
$\tau\colon  \widetilde X\to X$ as in
Proposition~\ref{invarians}, so that, for each structure form $\omega$ on
$X$,  
$\tau^*\omega$ is semi-meromorphic with a denominator that divides the
section $\eta$, and so that $\eta$ is locally
a monomial  in suitable  coordinates $s_j$. 

Let $\omega$ be a structure form on $X$. 
In this proof it is convenient to use the regularization
$$
R^f=\lim_{\epsilon\to 0}R^{f,\epsilon},  \quad 
R^{f,\epsilon}:=1-\chi(|f|^2/\epsilon)+\dbar \chi(|f|^2/\epsilon)\w u,
$$
where $u$ is the form \eqref{udefined} and 
$\chi$ is a smooth approximand of the characteristic function of $[1,\infty)$, cf., the beginning of
Section~\ref{svart}, so that  all the approximands
$R^{f,\epsilon}$ are smooth. 
If $f\equiv 0$ on a component $\widetilde X_j$ of $\widetilde X$, then
$R^{f,\epsilon}\equiv 1$ on $\widetilde X_j$ and if  $\phi$ satisfies
\eqref{bus} for any $\mu_0$, then $\phi\equiv 0$ on $\widetilde X_j$;
here we have suppressed  
the notation   $\tau^*$ for simplicity. Hence $\1_{\widetilde X_j} R^{f,\epsilon}\w\omega\phi=0$ and so $\1_{\widetilde X_j} R^{f}\w\omega\phi=0$. We can therefore assume that $f\not\equiv 0$ on $\widetilde X$. 
Thus the action of
$R^{f,\epsilon}\w\omega\phi$ on a test form is, via a partition of unity, a sum of integrals like 
$$
\int_{\widetilde  X} \frac{ds_1\w\cdots\w ds_n}{s_1^{\alpha_1+1}\cdots s_n^{\alpha_n+1}}\w 
R^{f,\epsilon}\phi\w \xi,
$$
where $\alpha_j$ are nonnegative integers and 
$\xi$ is a smooth form.  
Following  \cite[Section~3]{ASS} one can integrate by parts $|\alpha|:=|\alpha_1|+\cdots+|\alpha_n|$ times,
and get a constant times 
\begin{equation}\label{piktur}
\int_{\widetilde  X} \frac{ds_1\w\cdots\w ds_n}{s_1\cdots s_n}\w \partial^{\alpha}_s 
\big(R^{f,\epsilon}\phi\w\xi\big),
\end{equation}
where $\partial_s^\alpha=\partial^{|\alpha|}/\partial s_1^{\alpha_1}\cdots \partial s_n^{\alpha_n}$.

Let ut consider $\partial^\ell_s (R^{f,\epsilon} \phi)$. Assume that the
metric on $L$ is locally given so that $|f_j|^2=f_j\bar f_j a$, where $a$ is nonvanishing. Then 
\[
\sigma = \frac{\sum \bar f_j a e_j}{|f|^2},  
\]
cf.\ \eqref{farfar}, and so  
\begin{equation*} 
\frac{\partial}{\partial s_k} \sigma 
= 
\frac{\sum \bar f_j \frac{\partial a}{\partial s_k} e_j}{|f|^2} - 
\frac{\big (\sum \bar f_j a e_j\big ) 
\big (\sum \bar f_j a \frac{\partial f_j}{\partial s_k} \big ) }{|f|^4}  
=
\frac{1}{a}\frac{\partial a}{\partial s_k} \sigma - 
\big (\frac{\partial f}{\partial s_k}\cdot \sigma \big ) \sigma, 
\end{equation*}
i.e., ${\partial \sigma}/{\partial s_k}$ is of the form  
\begin{equation}\label{mint}  
\frac{\partial}{\partial s_k} \sigma = 
(\gamma\cdot \sigma ) \sigma, 
\end{equation}
where $\gamma$ is smooth. 
By iterated used of \eqref{mint}, since $\sigma\wedge\sigma =0$, we get that 
\begin{equation}\label{mynt} 
\partial_s^\kappa (\sigma\wedge (\dbar \sigma)^{k-1})=
 \sigma\wedge (\dbar \sigma)^{k-1} \wedge (\gamma_1\cdot \sigma)
 \wedge \cdots \wedge (\gamma_{|\kappa|}\cdot \sigma), 
\end{equation}
where $\gamma_j$ are smooth. 
If we take a smooth modification $\pi: \widehat X\to \widetilde X$
such that $\pi^*f=f^0f'$ as above, then  
$\pi^*\sigma=smooth /f^0$ 
and
thus $\pi^*(\partial^\kappa_{s} u)$ is like 
$1/ (f^0)^{\mu+|\kappa|}$. 
Moreover 
$\partial_s^\kappa \dbar \chi (|f|^2/\epsilon)$
is like $1/|f|^{|\kappa|+1}$ and with support where
$|f|^2\sim\epsilon$, see \cite{ASS}. Thus  $\partial^\kappa_s R^{f,\epsilon}$
is like $1/ |f|^{\mu+|\kappa|+1}$ and with support where
$|f|^2\sim\epsilon$; here we have suppressed $\pi^*$ for simplicity. 
Next, assume that $\mu_0\ge \mu +|\alpha|+1$ and that $\phi$
satisfies \eqref{bus}. Then by the smooth Brian\c con-Skoda theorem,
locally in $\widetilde X$, $\phi$ is in the ideal
$(f)^{\mu+|\alpha|+1}$, and 
therefore, 
$$
|\partial_s^\kappa\phi|\le C |f|^{\mu+|\alpha|-|\kappa|+1}. 
$$
Hence $\partial^\ell_s (R^{f,\epsilon} \phi)$ is bounded and with
support where
$|f|^2\sim\epsilon$ for $|\ell|\leq |\alpha|$. 
It follows 
by dominated convergence that \eqref{piktur}
tends to zero when $\epsilon\to 0$, cf.\ \cite[Section~4]{ASS}.

We finally choose $\mu_0$ so that $\mu_0 \geq n+ |\alpha|+1$ for all local representations
$\eta=s_1^{\alpha_1+1}\cdots s_n^{\alpha_n+1}$. Then
$R^f\w\omega\phi=0$ for all choices of $f$ if $\phi$ satisfies \eqref{bus}. 

\end{proof}


Note that the explicitness of $\mu_0$ in the proof above is directly
related to the explicitness of the modification $\tau:\widetilde X\to
X$. See \cite{Sz} for a complete description of the optimal $\mu_0$
in the case of plane curves.




\section{Proofs of Theorem~A  
and  variations}\label{helproofs}

Throughout this section we will use the notation from Theorem~A. 
For the proof of Theorem~A, besides the basic Lemma ~\ref{omljud}, we also need

\begin{lma}\label{nastanglatt}
Assume that $V\subset\C^N$ is smooth, and let $\omega$ be a structure form on $X$. 
Then there is a number $\mu'$ such
that $z_0^{\mu'} \omega$ is almost smooth on $ X$. 
\end{lma}

\begin{proof}
Let $\tau:\widetilde X\to X$ be as in
Proposition~\ref{invarians}. Then $\widetilde\omega:= \tau^*\omega$ is a semi-meromorphic form whose 
denominator locally is a monomial whose zeros are
contained in $\tau^{-1}X_{\text{sing}}$. 
Since $V$  is smooth, $X_{\text{sing}}\subset X_\infty\subset\{z_0=0\}$, and it follows that 
$\tau^*(z_0^{\mu'})\tau^*\omega$ is smooth for some large enough number
$\mu'$. Hence  $z_0^{\mu'}\omega$ is almost smooth.
\end{proof}

\begin{proof}[Proof of  Theorem~A]
Let $f_j$ be the $d$-homogenizations of $F_j$, let $R^f$ be the residue current constructed 
from the  Koszul complex $(E_\bullet^f, \delta_f)$ 
generated by $f_1,\ldots, f_m$,  
and  let $\phi$ be the $\rho$-homogenization of $\Phi$, 
with 
\begin{equation}\label{rhoval}
\rho=\max(\deg\Phi+(\mu+\mu_0)d^{c_\infty}\deg X, (d-1)\min(m,n+1)+\reg X),
\end{equation}
where $\mu_0$ is chosen as in Lemma~\ref{omljud}; in particular,
$\mu_0=0$ if $X$ is smooth.  
Note that $\mu_0$ 
only depends on $X$ and not
on the embedding $i:X\to \P^N$. 
Throughout this proof we will use the notation from Section~\ref{dist}.

The assumption \eqref{kaka} 
implies that $\blow^*\phi$ vanishes to
order $(\mu+\mu_0)r_j$ on each $W_j$ such that $\blow(W_j)$ is not contained
in $X_{\infty}$.
Now consider $W_j$ such that $\blow(W_j)\subset X_{\infty}$.
If $\Omega$ is a first Chern form for $\Ok(1)|_X$, e.g., $\Omega=dd^c\log|z|^2$, then $d\Omega$ is
a first Chern form for $L=\Ok(d)|_X$ on $X$ (notice that $d$ denotes  the degree and not the differential).
By \eqref{elest} we therefore have that
$$
r_j\int_{Z_j}(d \Omega)^{\dim Z_j}\le \int_X (d\Omega)^n, 
$$
which implies that
\begin{equation}\label{sunnanvind}
r_j\le d^{\codim Z_j}\deg X.
\end{equation}
By the choice \eqref{rhoval} of $\rho$, $\phi$ is of the form $z_0^{(\mu+\mu_0)d^{c_\infty}\deg X}$ times a 
holomorphic section, and thus
$\blow^*\phi$ 
vanishes to order at least $(\mu+\mu_0) r_j$
on $W_j$ for each $j$.
Hence \eqref{bus} holds, cf., \eqref{pulka}, and it follows from Lemma~\ref{omljud} that  $R^f\w\omega\phi=0$.

Since $\rho\ge (d-1)\min(m,n+1)+\reg X$ it follows from Lemma~\ref{direkt}
that we have a global $q$ such that $f\cdot q=\phi$ on $X$.
After dehomogenization we get a tuple of polynomials $Q_j$ such that \eqref{hummer}
holds and $\deg F_jQ_j\le\rho$. Thus  
part (i) of Theorem~A is proved.

\smallskip

For the second part choose
\[
\rho=\max (\deg\Phi+\mu d^{c_\infty}\deg X+\mu', (d-1)\min (m,
n+1)+\reg X),
\] 
where $\mu'$ is chosen as in Lemma~\ref{nastanglatt}, and let $\phi$ and $\phi'$ be the $\rho$- and
$(\deg\Phi+\mu d^{c_\infty}\deg X)$-homogenizations of $\Phi$, respectively.
Then, by Lemma~\ref{nastanglatt}, 
$$
 R^f\w\omega\phi= R^f\w \beta\phi',
$$
where $\beta$ is almost smooth, 
 and by \eqref{kaka2} and \eqref{sunnanvind},
\begin{equation}\label{prim}
|\phi'|\le C|f|^\mu. 
\end{equation}
Now take a smooth modification $\pi\colon \widetilde X\to X$ such that
$\beta=\pi_*\tilde\beta$, where $\tilde\beta$ is smooth, and
$f=f^0f'$, where $f^0$ is a section of a line bundle and $f'$ is nonvanishing.  Then $R^f\w\omega\phi$ is the push-forward under $\pi$
of a finite sum of  currents like
$$
(\pi^*\phi') \dbar\frac{1}{(f^0)^\mu}\w smooth,
$$
cf., \eqref{stuga1}, \eqref{stuga}, 
and in view of \eqref{prim} they must vanish.
Thus $ R^f\wedge \omega\phi=0$ and (ii) is proved as (i).
If $X$ is smooth even at infinity, then $\omega$ is smooth on $X$
so that we can choose $\mu'=0$ in Lemma~\ref{nastanglatt}.



\end{proof}

The statement in Remark~\ref{CM} follows as in the proof above, using 
Remark~\ref{indirekt}.


\begin{remark}
An alternative way of finding polynomials $Q_j$ such that
\eqref{hummer} holds would be to first solve the division problem $f\cdot
q=\phi$ on $X$ by means of Theorem~\ref{fund} and then extend the
solution to $\P^N$. This was indeed done in an earlier version of this 
paper, see \cite[Theorem~1.1]{AW11}.
The degree estimate so obtained coincides with 
\eqref{chimpans}, 
except that the last entry in the $\max$ 
is slightly different; in \cite[Section ~6]{AW11}, however, we show
that it is  bounded by $d\min(m,n+1) +\reg X-1.$ 
Thus, expressed in $\reg X$ the estimate in \cite{AW11} is somewhat 
less sharp than \eqref{chimpans}. 
Note that in \cite{AW11} we used the non-standard convention that
$\reg X$ is $\reg \S/J_X$ instead of $\reg J_X$, cf.\ Section ~\ref{lotta}.

\end{remark}

\begin{remark}\label{skola}
If
\begin{equation}\label{putt}
\codim (Z^f\cap X^\ell)\ge \mu+\ell+1,\quad \ell\ge 0,
\end{equation}
where $X^\ell$ are as in Section~\ref{fundform}, thus either
$X_{\text{sing}}\cap Z^f=\emptyset$ or $m<n$, then one can find
polynomials $Q_j$ such that \eqref{hummer} holds and \eqref{chimpans}
holds with $\mu_0=0$. 
To see this,  take $\rho\ge \deg\Phi +\mu d^{c_\infty}\deg X$ in the proof of
Theorem~A. Then $R^f\phi=0$ on $X_{\text{reg}}$, and thus $R^f\wedge \omega\phi$ has support on $Z^f\cap X^0$. 
Since $R^f\wedge \omega_0\phi$ has bidegree at most $(n,\mu)$ and $\codim (Z^f\cap X^0)\geq \mu+1$ by 
\eqref{putt}, it follows from the dimension  principle that $R^f\w\omega_0\phi=0$. 
Thus
$ R^f\w\omega_1\phi= R^f\w\alpha^1\omega_0\phi$ vanishes outside $X^1$, so again by
\eqref{putt} and the dimension principle we find that
$ R^f\w\omega_1\phi$ vanishes identically. By induction,
$ R^f\w\omega\phi=0$.
\end{remark}

\begin{ex} 
In light of the following example due to Masser, Philippon,
Brownawell, and Koll\'ar, see \cite[page~578]{Brown} or 
\cite[Example~2.3]{Koll}, one can see that
the power $c_{\infty}$ in Theorem~A cannot be improved: 
Let $X=\P^n$ and let $m$ be an integer with $2\le m\le n$. Consider the
$m$ polynomials
$$
z_1^d, \ z_1 z_m^{d-1}-z_2^d,\ldots, z_{m-2}z_m^{d-1}-z_{m-1}^d,\
z_{m-1}z_m^{d-1}-1,
$$
in $\C^n$. The associated projective variety
$\{z_0=z_1=\cdots =z_{m-1}=0\}\subset X_\infty$ has  codimension $m$,  and
hence $c_\infty=m$, cf., \eqref{bokslut}.
It follows from Theorem~A that we have a representation \eqref{hummer}
with  $\Phi=1$ and  $\deg F_jQ_j\le m d^m$ (if $d$ is not too small).
However, if $Q_j$ are any polynomials so that \eqref{hummer} holds with $\Phi=1$, then
by considering the curve
$$
t\mapsto(t^{d^{m-1}-1}, t^{d^{m-2}-1},\ldots,
t^{d-1}, 1/t,0,  \ldots, 0),
$$
one can conclude that $Q_1$ must have degree at least $d^m-d$ so that
$\deg F_1Q_1 \ge d^m$.
\end{ex}

\begin{remark}\label{gbh}
In the proof above $\mu_0$ is derived from the section
$\eta$ in Proposition ~\ref{invarians}. Since we have a fixed embedding
 $X\to\P^N$ we can get a slighly sharper constant $\mu_0'$.  
In fact, if $\omega'$ is an associated structure form we can replace
$\eta$ in the proof by a section $\eta'$ such that 
$\eta'\tau^*\omega'$ is smooth, cf., Remark ~\ref{torsdag}. If $A'$ is the highest
degree of the (in local coordinates) monomial $\tau^*\eta'$  then $\mu_0'=1+A'$.
If $A$ is the maximal degree of $\tau^*\eta$ then $A\le (n-1)A'$. It follows
that $\mu_0\le(n-1)\mu_0'$ so we can however gain at most a factor $n-1$ by considering
the special embedding.
\end{remark}

In \cite{A8} is used   a slight generalization of the Koszul complex 
to deal with
a positive power  $\J^\ell_f$ of $\J_f$, cf. \cite[p.~439]{EL}; this complex is a special case of the Eagon-Northcott complex, see,
e.g., \cite[Appendix~2.6]{Eis}. 
The first mapping in the complex is the natural mapping
$
E^{\otimes \ell}\to \C
$
induced by the $f_j$.
The associated residue current is the push-forward of currents like
$$
\dbar\frac{1}{(f^0)^{k}}\w  smooth
$$
for $\ell\leq k\le\mu+\ell-1$.  By an analogous proof we  get the  following generalization of
Theorem~A.

\begin{thm}\label{borta}
With the notation in Theorem~A, if
$$
|\Phi|/ |F|^{\mu+\mu_0+\ell-1} \text{ is locally bounded on } V,
$$
 then $\Phi\in(F_j)^\ell$ and there are polynomials  $Q_I$ such that
$$
\Phi=\sum_{I_1+\cdots+I_m=\ell} F_1^{I_1}\cdots F_m^{I_m} Q_I
$$
and
\begin{multline*}
\deg (F_1^{I_1}\cdots F_m^{I_m} Q_I)\leq \\
\max \big(\deg \Phi + (\mu+\mu_0+\ell-1) d^{c_\infty} \deg X, d(\min(m,n+1)+\ell-1) -\min(m,n+1)+\reg X \big).
\end{multline*}
\end{thm}

There is also an analogous generalization  of part (ii) of Theorem~A.



\section{Proofs  of  Theorem~B and  variations}\label{ormbo}

We first look at  the case when $X$ is smooth, which is due to
Ein-Lazarsfeld \cite{EL}.

\begin{thm}\label{elthm}
Let $X$ be a smooth projective  variety,
let $L\to X$ be an ample Hermitian line bundle, and let 
$A\to X$ be a line bundle that is either ample or big and nef.
Moreover, let $f_1,\ldots,f_m$ be global holomorphic sections of $L$,
and let $\phi$ be a global section of $$
L^{\otimes s}\otimes K_X \otimes A,
$$
where  $s\ge \min(m,n+1)$. If
\begin{equation}\label{bs}
|\phi|\le C |f|^\mu
\end{equation}
on $X$, then there are holomorphic sections $q_j$ of $L^{\otimes (s-1)}\otimes K_X \otimes A$
such that
\begin{equation}\label{likhet}
f_1q_1+\cdots +f_mq_m=\phi.
\end{equation}
\end{thm}

Let $\J_f$ be the ideal sheaf generated by $f_j$ and assume that
the associated distinguished varieties $Z_k$  have multiplicities $r_k$, cf., Section~\ref{dist}.
If $\phi$ vanishes to (at least) order $r_k\mu$ at a generic point on $Z_k$
for each $k$, then
\eqref{bs} holds, cf., Section~\ref{dist}, and thus we have

\begin{cor}
If $\phi$ vanishes to  order $r_k\mu$ at a generic point on $Z_k$,
for each $k$, then we have a representation
\eqref{likhet}.
\end{cor}

This corollary is precisely part (iii) of the main theorem in \cite[p.\ 430]{EL}, 
except for that  we
have $\mu r_k$ rather than $(n+1)r_k$, cf., the discussion in
Example~\ref{uppstuds}. 
Using  \eqref{elest} one gets the estimate $r_k\le \deg_L X$.


\begin{proof}[Proof of Theorem~\ref{elthm}]
Let $(E^f_\bullet, \delta_f)$ be the
Koszul complex generated by $f_1,\ldots, f_m$, as in Example~\ref{koszulex},
tensorized with $L^{\otimes s}\otimes A\otimes K_X$,  and
let $R^f$ be the associated residue  current on $X$.
From the hypothesis  \eqref{bs} and Lemma~\ref{omljud} we conclude  that  $R^f\phi=0$.
The bundle $E_k$ is a direct sum of line bundles
$L^{\otimes(s-k)}\otimes A\otimes K_X$ and so all the relevant
cohomology groups \eqref{cohovillkor} vanish  by Kodaira's vanishing theorem,
or, at the top degree, by the  Kawamata-Viehweg  vanishing theorem if $A$ is nef and big.
Thus  Theorem~\ref{elthm} follows from Proposition~\ref{glatta}.
\end{proof}

\begin{proof}[Proof of Theorem~B]

Let $(E_\bullet^f, \delta_f)$ be the
Koszul complex generated by $f_1,\ldots, f_m$ 
tensorized with $L^{\otimes s}$, see Example \ref{koszulex}.
The choice of $s$ guarantees that \eqref{cohovillkor} is satisfied and thus by Theorem ~\ref{fund}~(ii) we get the desired holomorphic solution to \eqref{likhet2} as soon as we have a smooth solution to 
\begin{equation}\label{saattdee}
\nabla_f W = \phi
\end{equation}
on $X$. Indeed, recall that Theorem ~\ref{fund}~(ii) only depends on
$X$ and not on the embedding $i:X\to Y$.  
Hence to prove the theorem it suffices to show that there is a $\mu_0$
such that we can find a smooth solution to \eqref{saattdee} for each
global section $\phi$ of $L^{\otimes s}$ that satisfies \eqref{bs2}. 
As in the proof of Theorem~A the strategy will be to show that $\phi$
annihilates a certain residue current, which gives a smooth solution
to \eqref{saattdee}. 
Note that we cannot apply Theorem~\ref{fund}~(i), since a~priori $L$ and the sections $f_j$ are only
defined on $X$.

Let us start by giving an overview of the proof below. First, there is an
embedding of $X$ into a smooth
manifold $Y$ so that $L$ extends to $Y$. We cannot assume
that $f$ extends holomorphically to $Y$ but in view of Lemma
~\ref{ut}, if $\big (\Ok (E^h_\bullet), h\big )$ is a
Hermitian resolution of  $\Ok^Y/\J_X$, then there is a $\nabla_h$-closed extension $\tilde f$. Given $\tilde f$ we
construct a Koszul complex $(E_\bullet^H\otimes \Lambda^\bullet E,
\delta_{\tilde f})$ that extends
$(E_\bullet^f,\delta_f)$, and following the ideas in Example ~\ref{koszulex}  we
construct a residue current $\widetilde R\wedge R^g$, where
$\big (\Ok(E_\bullet^g),g\big )$ is again a resolution of $\Ok^Y/\J_X$. 
From the
construction it follows that 
if 
\begin{equation}\label{strunt}
\widetilde R\wedge R^g\phi =0,  
\end{equation} 
then there is a current solution to 
\begin{equation}\label{ljusgul}
\nabla W =\Phi,
\end{equation} 
 where $\nabla=g+\delta_{\tilde f}+h-\dbar$. 
From such a solution we obtain a smooth solution to \eqref{ljusgul}, which in turn implies
that there is a smooth solution to \eqref{saattdee}. 
Finally we show that there is a $\mu_0$, only depending on $X$, such
that \eqref{strunt}, and thus \eqref{saattdee},  holds as soon as $\phi$ satisfies \eqref{bs2}.

\medskip 

We first discuss the extension of $L$. 
If  $M$ is large enough, there are embeddings $i_j\colon X\to\P^{N_j}$, $j=1,2$,  such that
$\Ok(1)_{\P^{N_1}}|_X=L^M$ and $\Ok(1)_{\P^{N_2}}|_X=L^{M+1}$.
If $\pi_j\colon \P^{N_1}\times\P^{N_2}\to \P^{N_j}$, then
$\L:=\pi_2^*\Ok(1)_{\P^{N_2}}\otimes \pi^*_1\Ok(-1)_{\P^{N_1}}$ is a line bundle
over $Y:=\P^{N_1}\times\P^{N_2}$
and its restriction to  $X\simeq \Delta_{X\times X}\subset Y$ is precisely $L$.  This argument
was communicated to us by  R.\ Lazarsfeld.

Let $\big (\Ok(E_\bullet^h), h\big )$ be a Hermitian resolution of $\Ok^Y/\J_X$ in
$Y$ as in Section \ref{fundform}. 
In view of Lemma~\ref{ut},  we can choose smooth $\nabla_h$-closed extensions
$\tilde f_j\in\oplus_i\E_{0,i}(E^h_i\otimes \L)$  of $f_j$ to $Y$, as
defined in Section~\ref{divandres}. 
Let $E^1,\ldots, E^m$ be (extensions to $Y$ of) the trivial line bundles used to define
$(E_\bullet^f, \delta_f)$ as in Example~\ref{koszulex}, with basis elements $e_1, \ldots, e_m$, respectively, and let $\tilde f$ be the section $\tilde f:=\tilde f_j e_j^*$ of $E^h_\bullet\otimes E^*$, where $E:=\bigoplus_{j=1}^m \L^{-1}\otimes E^j$ and $e_j^*$ are the dual basis elements. Note that each $\tilde f_j$ has even degree so that $\tilde f$ has odd degree. 

\smallskip 

We next want to construct a 
Koszul complex of $\tilde f$ as an extension of $(E^f_\bullet,
\delta_f)$. 
To this end we will need to take products of sections of
$E^h_\bullet$. We therefore introduce $E^H_\bullet:= \bigcup_{k\geq 1}
(E^h_\bullet)^{\otimes k}$, where the tensor products 
$(E^h_\bullet)^{\otimes k}$ are as in Section ~\ref{tenprod}. Since
$E^h_0$ is the trivial line bundle, $(E^h_\bullet)^{\otimes k}$ is a
natural subcomplex of $(E^h_\bullet)^{\otimes (k+1)}$ and thus the
definition makes sense. 
Next consider the tensor product complexes $E^H_\bullet\otimes
\Lambda^{k}E$, see Section \ref{tenprod} and let $\delta_{\tilde f}: E^H_\bullet\otimes \Lambda^{k}E \to E^H_\bullet\otimes \Lambda^{k-1} E$ be contraction with $\tilde f$, i.e., 
for a section $\xi=\sum_{I=\{ i_1,\ldots , i_k\}}\xi_I\cdot e_I$,
where $\xi_I$ takes values in $E^H_\bullet$ and $e_I=e_{i_1}\wedge
\cdots\wedge e_{i_k}$, of $E^H_\bullet\otimes \Lambda^k E$,
$\delta_{\tilde f}\xi = \sum_I (-1)^{\deg \xi_I} \xi_I \sum _j
(-1)^{j-1} \tilde f_{i_j} \cdot e_{I\setminus i_j}$. 
Note that $\delta_{\tilde f}$ is an anti-derivation. 
As long as we restrict to $X$ we can write  $\tilde f= f-f'$, where $f:= \sum f_j e_j^*$ and 
$f'$ has positive form  degree. Let $\delta_f$ and $\delta_{f'}$  
be defined analogously to $\delta_{\tilde f}$;  
note that, restricted to $X$, $\delta_f$ is just the differential in the regular Koszul
complex $(E_\bullet^f, \delta_f)$. 

\smallskip 
 
Inspired by Example ~\ref{koszulex} we now construct the
residue current $\widetilde R\wedge R^g$. We start by defining an  
$(E^H_\bullet\otimes \Lambda^\bullet E)$-valued form
$\tilde u$ which will play the role of $u$; in fact $\tilde u$
will take values in $(E^h_\bullet)^{\otimes n}\otimes \Lambda^k E$. 
First, let
$\sigma$ be the section of $E$ over $X\setminus Z$ of pointwise minimal norm such that $\delta_f\sigma=1$ there, cf. Example~\ref{koszulex}. 
Then
$$
\delta_{\tilde f}\sigma=\delta_f \sigma-\delta_{f'}\sigma=1-\delta_{f'}\sigma
$$
on $X\setminus Z$. Notice that $\delta_{f'} \sigma$ has even degree, and form bidegree at least $(0,1)$, so that
$$
\frac{1}{1-\delta_{f'}\sigma}=1+\delta_{f'}\sigma+(\delta_{f'}\sigma)^2+\cdots +(\delta_{f'}\sigma)^n
$$
is a form on $X\setminus Z$ with values in $E^H_\bullet\otimes \Lambda^\bullet E$. 
Let  
$
\tilde\sigma:=\sigma/(1-\delta_{f'}\sigma)
$
on $X\setminus Z$; then
$
\delta_{\tilde f}\tilde\sigma=1
$ 
on $X\setminus Z$.   
Next, let  
$$
\tilde u=\frac{\tilde \sigma}{(\delta_{\tilde f}+\nabla_h)\tilde\sigma}=
\sum_{k\geq 1} \tilde\sigma\wedge(-\nabla_h \tilde\sigma)^{k-1},  
$$
cf., Example~\ref{koszulex} and \cite{A2}; now $\nabla_h$ plays the
role of $-\dbar$. 
Note that $\delta_{\tilde f}$ anti-commutes with (the extension to $E^H_\bullet\otimes \Lambda^\bullet E$ of) $\nabla_h$, i.e., $\delta_{\tilde f}\circ \nabla_h=-\nabla_h\circ \delta_{\tilde f}$. 
It follows that $(\delta_{\tilde f}+\nabla_h)^2=0$ and so  
$$
(\delta_{\tilde f}+\nabla_h)\tilde u=1
$$
on $X\setminus Z$, cf. Section~\ref{olga}.

Let $\big (\Ok(E_\bullet^g),g\big )$ be a Hermitian resolution
of $\Ok^Y/\J_X$ in $Y$, let $R^g$ be the residue current associated with the resolution
$\big (\Ok(E^g_\bullet), g\big )$, and let $\omega$ be the associated structure form. 
Recall from Section~\ref{fundform} that if $\alpha$ is a
(sufficiently) 
smooth form on $X$, then $\alpha\wedge R^g$ is a well-defined current in $Y$; in particular,  $\chi(|f|^{2}/\epsilon) \tilde u\w R^g$ is a well-defined current in $Y$ with values in $E^H_\bullet\otimes \Lambda^\bullet E\otimes E^g_\bullet$. 
Letting  
\begin{equation}\label{messa}
\nabla=g+\delta_{\tilde f}+\nabla_h=g+\delta_{\tilde f}+h-\dbar, 
\end{equation}
note that $\nabla^2=0$ and that 
\begin{equation}\label{bengalo}
\nabla\big (\chi(|f|^2/\epsilon\big )\tilde u\w R^g+U^g)=I-\tilde R^\epsilon\w R^g,
\end{equation}
where 
$
\tilde R^\epsilon=I-\chi(|f|^2/\epsilon)I+\dbar\chi(|f|^2/\epsilon)\wedge\tilde u.
$

We claim that $\chi(|f|^2/\epsilon)\tilde u\w R^g$ has a limit when $\epsilon\to 0$. 
To see this, recall from Section~\ref{fundform}, using the notation from that section, that $\chi(|f|^2/\epsilon)\tilde u\w R^g\w\varOmega=i_*(\chi(|f|^2/\epsilon)\tilde u\w\omega)$. 
Next, notice that
\begin{equation}\label{pladder}
\tilde\sigma\w(-\nabla_h\tilde\sigma)^{k-1}=\sigma\w(\dbar\sigma)^{k-1}\wedge 
\sum_{j=0}^n c_j^k (\delta_{f'} \sigma)^j,
\end{equation}
for some numbers $c_j^k$, since $\sigma\w\sigma=0$ and $\sigma$ 
has degree $0$ in $E_\bullet^h$.  
Let $\pi:\widetilde X\to X$ be a smooth modification such that $\pi^*\omega$ is semi-meromorphic and $\pi^*\sigma$ is of the form $\sigma'/f^0$, cf. Section~\ref{dist}. 
Then $\pi^*\tilde u$ is a finite sum of terms $\gamma_k/(f^0)^k$,
where $\gamma_k$ are smooth, and hence  
$\lim_{\epsilon\to 0}\pi^*(\chi(|f|^2/\epsilon)\w\tilde u\w\omega)$
exists, see, e.g., \cite{BjSam}. Since $\varOmega$ is nonvanishing it follows that the limit of $\chi(|f|^2/\epsilon)\tilde u\w R^g$ exists.

Let 
$$
\tilde U\w R^g=\lim_{\epsilon\to 0}\chi(|f|^2/\epsilon)\tilde u\w R^g, 
\quad
\tilde  R\w R^g=\lim_{\epsilon\to 0}\tilde R^\epsilon\w R^g.
$$
Then by \eqref{bengalo}, 
$$
\nabla(\tilde U\w R^g+U^g)=I-\tilde R\w R^g,   
$$
and if $\Phi$ is a smooth $\nabla_g$-closed extension of $\phi$ as in
Lemma ~\ref{ut} (regarded as a section of  $\L^{\otimes s}\otimes E^H_\bullet\otimes \Lambda^\bullet E\otimes E^g_\bullet$), 
it follows that
\begin{equation}\label{gurka1}
\nabla\big((\tilde U\w R^g+U^g)\w\Phi\big )=\Phi
\end{equation}
in $Y$ as soon as \eqref{strunt} is satisfied, 
since, as was noted in the proof of Lemma~\ref{ut},  $R^g\w\Phi= R^g\phi$.


\smallskip


By a slight modification of Proposition~\ref{urlaka}, if we have a
current solution to \eqref{ljusgul} we also have a smooth solution. 
To see this, let $E^F_\bullet=\Lambda^\bullet E\otimes E^g_\bullet$ and let $\M_\bullet$ and $\M^\E_\bullet$ be defined as in Section~\ref{sheafcomplex}, but for the complex $E^H_\bullet$ instead of $E^f_\bullet$. Then we have the double complex   
$$
\Be_{\ell,k}:=\oplus_j\Cu_{0,j}(E^H_{j+k}\otimes E_\ell^F)=\M_k(E^F_\ell)
$$
with mappings $\nabla_h\colon \Be_{\ell,k}\to\Be_{\ell,k-1}$ and $F:=g+\delta_{\tilde f}\colon\Be_{\ell,k}\to\Be_{\ell-1,k}$; indeed note that $\nabla_h\circ F=-F\circ \nabla_h$. If $\Be_j:=\bigoplus_{\ell+k=j}\Be_{\ell, k}$ we get the associated total complex
$$
\ldots 
\stackrel{\nabla}{\longrightarrow} \Be_j \stackrel{\nabla}{\longrightarrow}
\Be_{j-1}\stackrel{\nabla}{\longrightarrow}\ldots ~~~~~,  
$$
with $\nabla$ as in \eqref{messa}. 
Analogously let 
$\Be^\E_{\ell,k}:=\oplus_j\E_{0,j}(E^H_{j+k}\otimes
E_\ell^F)=\M^\E_k(E^F_\ell)$ with total complex
$\Be^\E_\bullet$. Moreover, let $\M_\bullet(Y,E_\ell^F)$,
$\M_\bullet^\E(Y,E_\ell^F)$, $\Be_\bullet(Y)$, and $\Be^\E_\bullet(Y)$ be the associated complexes of global sections. 
Note that we have natural mappings 
\begin{equation}\label{internet}
H^j(\Be^\E_\bullet(Y))\to H^j(\Be_\bullet(Y)),  \quad j\in \Z.
\end{equation}
Proposition~\ref{urlaka} implies that the natural mappings 
$H^k(\M^\E_\bullet(Y,E^F_\ell))\to H^k(\M_\bullet(Y,E^F_\ell))$ are
isomorphisms. Now, by repeating the proof of Proposition~\ref{urlaka}
with $\M_\bullet$, $\M^\E_\bullet$, $\Cu_{0,\bullet}$, and
$\E_{0,\bullet}$ replaced by $\Be_\bullet$, $\Be^\E_\bullet$,
$\M_\bullet$, and $\M^\E_\bullet$, respectively, using that the double
complex $\Be_{\ell, k}$ is bounded in the $\ell$-direction, we can
therefore prove that the mappings \eqref{internet} are in fact
isomorphisms. Hence the current solution \eqref{gurka1} gives a smooth
solution to \eqref{ljusgul}. 

Next we will show that a smooth solution to \eqref{ljusgul} gives a
smooth solution to \eqref{saattdee}. 
Let lower indices $(i, j, k)$ denote components in $\L^{\otimes
  s}\otimes E^H_i\otimes \Lambda^j E\otimes E^g_k$. Then $\Phi=\Phi_{0,0,0}+\Phi_{0,0, 1}
+ \cdots + \Phi_{0,0,n}$, where $\Phi_{0,0,k}$ has form bidegree
$(0,k)$. Notice that we have the decomposition $\tilde f=f_0-f'$ in
$Y$, where $f_0$ denotes the $0$-component of $\tilde f$ and hence is
a smooth extension of $f$ to $Y$.  If $\Psi$ is a smooth solution to
\eqref{ljusgul} it follows that 
\begin{eqnarray}
\label{kalla}
h \Psi_{1,0,0}+\delta_{f_0} \Psi_{0,1,0}+ g\Psi_{0,0,1} &= &\Phi_{0,0,0}, 
\\\label{skalla} 
\quad h\Psi_{1,j,0}+\delta_{f_0}\Psi_{0,j+1,0}+
g\Psi_{0,j,1}-\dbar\Psi_{0,j,0}& = & 0,\ j\ge 1. 
\end{eqnarray}
Indeed, note that $\delta_{f'}\Psi_{i,j,k}$ has positive degree in $E^H_\bullet$ for all nonvanishing $\Psi_{i,j,k}$.  
Since $\Psi$ is smooth, we can define the smooth forms
$W_j:=i^*\Psi_{0,j,0}$ on $X$. Note that $\L^{\otimes s}\otimes
\Lambda^jE|_X=E^f_j$, so that $W_j$ takes values in $E^f_j$. 
Since $g \Psi_{0,j,1}=g^1 \Psi_{0,j,1}$ and $h \Psi_{1,j,0}=h^1
\Psi_{1,j,0}$ are in $\E(\J_X)$ and thus vanish on $X$, 
\eqref{kalla} and \eqref{skalla} implies 
$$
\delta_f W_1=\phi, \quad  \delta_f W_{j+1}-\dbar W_j=0, \ j\ge 1.  
$$


\smallskip 

To sum up so far, we have shown that there is a smooth solution to
\eqref{saattdee} if \eqref{strunt} holds. 

\noindent 
\textbf{Claim}:  
\emph{There is a $\mu_0$, only depending on $X$, such that
\eqref{strunt} holds as soon as $\phi$ satisfies \eqref{bs2}.}

\noindent
Taking the claim for granted we get that 
there is a $\mu_0$ such that if 
$\phi$ satisfies \eqref{bs2}, then there is a smooth solution to
\eqref{saattdee}; this concludes the proof of Theorem~B.

The claim is essentially Lemma ~\ref{omljud}, but now $R^f$ is replaced
by the current $\widetilde R$. Also the proof is analogous to the proof
of the lemma and the choice of $\mu_0$ in the proof of the lemma will do here as
well; the crucial 
observation is that the singularities of $\widetilde R$ can be
controlled in a similar way to the singularities of $R^f$.

To prove the claim, first note that \eqref{strunt} is equivalent to that $\tilde R\w R^g
\w\varOmega\phi=\lim_{\epsilon\to 0} i_*(\tilde R^\epsilon
\w\omega\phi)$ vanishes, cf.\ Section ~\ref{boork}. 
Let $\tau:\widetilde X\to X$ be a smooth modification as in
Proposition~\ref{invarians}, so that locally
$\tau^*\omega={smooth}/{s^{\alpha+1}}$, where $s^{\alpha+1}$ is a
local representation of the section $\eta$, as in the proof of Lemma~\ref{omljud}. 
Following that proof, the action of $\tilde R^\epsilon\w\omega\phi$  on a test form is a sum of integrals like (suppressing $\tau^*$ for simplicity) 
\begin{equation}\label{nabb}
\int_{\widetilde  X} \frac{ds_1\w\ldots\w ds_n}{s_1\cdots s_n}\w \partial^{\alpha}_s 
\big(\tilde R^{\epsilon}\phi\w\xi\big),
\end{equation}
where $\xi$ is smooth, cf.\ \eqref{piktur}. The components of
$\widetilde X$ where $f$ vanishes identically are taken care of as in
the proof of Lemma \ref{omljud}. We may therefore assume that
$f\not\equiv 0$.  


In view of \eqref{pladder}, $\widetilde R^\epsilon$ is a finite sum of terms like
$$
\dbar\chi(|f|^2/\epsilon)\w \sigma \w(\dbar\sigma)^{k-1}\w
(\delta_{f'} \sigma)^j,
$$
where $k+j\le n$ for degree reasons; indeed, recall that $f'$ has form degree
at least $(0,1)$. 
Note that $\delta_{f'}\sigma$ is of the form $(\gamma\cdot \sigma)$,
where $\gamma$ is smooth. 
Thus by arguments as in the proof of Lemma~\ref{omljud}, cf. \eqref{mint} and
\eqref{mynt}, 
we get that  
$\partial^\ell_s\tilde R^\epsilon$ is like
$1/|f|^{n+|\ell|+1}$ and with support where $|f|^2\sim \epsilon$. 
As in that proof, we choose $\mu_0$ so that $\mu_0\geq n+ |\alpha|+1$ for all local
representations $\eta=s^{\alpha+1}$. If $\phi$ satisfies \eqref{bs2},
then $|\partial^\ell_s \phi|\leq C|f|^{n+|\alpha|-|\ell|+1}$, cf.\ the proof of Lemma~\ref{omljud}. Now
by dominated convergence \eqref{nabb} tends to zero when $\epsilon\to
0$, and since the choice of $\mu_0$ only depends on the section $\eta$
and $n$ and not on the
embedding $i:X\to Y$, the resolutions $\big ( \Ok(E^h_\bullet),
h\big )$, $\big ( \Ok(E^g_\bullet), g\big )$ or the
extension $\tilde f$, 
the claim follows.

\end{proof}


\begin{remark} If $(E^h_\bullet ,h)$ is a Koszul complex, then we just
simply take $E^H_\bullet=E^h_\bullet$, since he desired "product"
already exists within $E^h_\bullet$.

\end{remark}

\begin{remark}

Assume that $i:X\to Y$ is an embedding
such that all ample line bundles on $X$ extend to $Y$. Following the
proof above we  
then obtain Theorem~B without relaying on the quite involved
Proposition ~\ref{invarians}, since we can then define the $\mu_0$
from the singularities of one fixed structure form, cf.\ Remark ~\ref{gbh}.  
It is of course enough that there is a finite number of embeddings
of $X$ into smooth manifolds such that
each ample line bundle extends to at least one of them.

\end{remark}

In analogy with Theorem ~\ref{borta} we also have the following generalizations
of Theorem ~\ref{elthm} and Theorem ~B.

\begin{thm} With the notation in Theorem~\ref{elthm}, if $\phi$ is a section of
$L^{\otimes s}\otimes K_X\otimes A$, where $s\ge \min(m,n+1)+\ell-1$, and
$$
|\phi|\le C|f|^{\mu+\ell-1},
$$
then there are holomorphic sections $q_I$ of $L^{\otimes (s-\ell)}\otimes K_X\otimes A$, such that
\begin{equation}\label{bengan}
\phi=\sum_{I_1+\cdots + I_m=\ell} f_1^{I_1}\cdots f_m^{I_m}q_I.
\end{equation}

With the notation in Theorem~B, if  $\phi$ is a section of $L^{\otimes s}$
with $s\ge \nu_L+\min(m,n+1)+\ell-1$ such that
$$
|\phi|\le C|f|^{\mu_0+\mu+\ell-1},
$$
then there are holomorphic sections $q_I$ of $L^{\otimes (s-\ell)}$ such that \eqref{bengan} holds.
\end{thm}

\def\listing#1#2#3{{\sc #1}:\ {\it #2},\ #3.}

\end{document}